\documentclass[11pt]{amsart}
\usepackage[a4paper,margin=1in]{geometry}
\usepackage{amsmath,amssymb,amsthm,mathtools}
\usepackage[T1]{fontenc}
\usepackage{lmodern}
\usepackage{microtype}
\usepackage[colorlinks=true,linkcolor=blue,citecolor=blue,urlcolor=blue]{hyperref}
\usepackage{enumitem}
\usepackage{booktabs}
\usepackage{longtable}
\usepackage{array}
\usepackage{algorithm}
\usepackage{algpseudocode}

\newtheorem{theorem}{Theorem}[section]
\newtheorem{proposition}[theorem]{Proposition}
\newtheorem{lemma}[theorem]{Lemma}
\newtheorem{corollary}[theorem]{Corollary}
\newtheorem{conjecture}[theorem]{Conjecture}
\theoremstyle{definition}
\theoremstyle{remark}\newtheorem{remark}[theorem]{Remark}

\newcommand{\Pic}{\operatorname{Pic}}
\newcommand{\Isom}{\operatorname{Isom}}
\newcommand{\Aut}{\operatorname{Aut}}
\newcommand{\CAut}{\operatorname{CAut}}
\newcommand{\rk}{\operatorname{rk}}
\newcommand{\Lie}{\operatorname{Lie}}
\newcommand{\tr}{\operatorname{tr}}
\newcommand{\Sym}{\operatorname{Sym}}
\newcommand{\Sp}{\operatorname{Sp}}
\newcommand{\ch}{\operatorname{ch}}
\newcommand{\Herm}{\operatorname{Herm}}
\newcommand{\ind}{\operatorname{ind}}
\newcommand{\td}{\operatorname{td}}
\newcommand{\SO}{\operatorname{SO}}
\newcommand{\End}{\operatorname{End}}
\newcommand{\vol}{\operatorname{vol}}
\newcommand{\Ext}{\operatorname{Ext}}
\newcommand{\Hom}{\operatorname{Hom}}
\newcommand{\CC}{\mathbf C}
\newcommand{\RR}{\mathbf R}

\newcommand{\HH}{\mathbf H}
\newcommand{\EE}{\mathbf E}
\newcommand{\Ahat}{\widehat A}

\newcommand{\CertificateRepository}{https://github.com/agohr/ls_dim_14}

\title{The LeBrun--Salamon conjecture is true in dimensions up to $56$}

\author{Jordi Daura Serrano}
\address{J. Daura Serrano: Department of Mathematics and Statistics, University of Cyprus, Nicosia 1678, Cyprus}
\email{daura-serrano.jordi@ucy.ac.cy}

\author{Aron Gohr}
\address{Aron Gohr: Department of Mathematics, South Kensington Campus,
Imperial College London, 
London, SW7 2AZ, UK}
\email{aron.gohr@gmail.com}

\author{Marie-Am\'{e}lie Lawn}
\address{M.-A.~Lawn: Department of Mathematics, 
South Kensington Campus,
Imperial College London, 
London, SW7 2AZ, UK}
\email{m.lawn@imperial.ac.uk}

\author{Travis Schedler}
\address{T. Schedler: Department of Mathematics, South Kensington Campus,
Imperial College London, 
London, SW7 2AZ, UK}
\email{t.schedler@imperial.ac.uk}

\date{5 October 2026}


\begin{document}

\begin{abstract}
We prove the LeBrun--Salamon conjecture for positive quaternionic K\"ahler manifolds of quaternionic dimension $n\le 14$.
To do so, we first show that the isometry group $\Isom(M)$ of any minimal counterexample $M$ has rank at most one, hence dimension at most three.  On the other hand, we consider a characteristic number  which equals $\dim \Isom(M) - \delta_n$, where $\delta_n = n+3$ if $n$ is even and $\delta_n = n+1$ if $n$ is odd, and show that this index is a sum of positive quantities when $n \leq 14$, yielding a contradiction.

The toolset used to show positivity enlarges in stages. Projection curvature moments suffice through $n=8$; positive-semidefinite curvature moments extend the
argument through $n=10$; and signed scalar orbital moments give the new inequalities needed for $n=11,12$. Scalar orbital moments alone fail at $n=13$; combining them with Hodge--Riemann positivity for squares of cohomology classes extends the proof to $n=13,14$.
We give exact separating obstructions explaining the changes of method.  
\end{abstract}
\maketitle

\section{Introduction}

\subsection{The LeBrun--Salamon conjecture}
Quaternionic K\"ahler manifolds form one of the distinguished classes in Berger's classification of irreducible Riemannian holonomy groups. A Riemannian manifold $(M^{4n},g)$, with $n\geq2$, is called quaternionic K\"ahler if its holonomy group is contained in $\Sp(n)\Sp(1)\subseteq\SO(4n)$, or equivalently, if there is a rank three subbundle $Q \subseteq \End(TM)$ of endomorphisms, parallel for the Levi-Civita connection, which together with the identity span quaternion algebras $\mathbb{H}$ at every point. Note that in dimension four this holonomy condition is vacuous, since $\Sp(1)\Sp(1)\cong\SO(4)$. One instead requires $g$ to be Einstein and the conformal curvature to be self-dual; equivalently, we require the almost complex structure on the twistor space of $M$ to be integrable.

Quaternionic K\"ahler manifolds are Einstein, and those with positive scalar curvature are called positive quaternionic K\"ahler manifolds. The compact simply connected quaternionic K\"ahler symmetric spaces were classified by Wolf \cite{Wolf1965} and are known as Wolf spaces. They provide all currently known examples of complete positive quaternionic Kähler manifolds. Indeed, the LeBrun--Salamon conjecture asserts: 

\begin{conjecture}
Every connected complete positive quaternionic K\"ahler manifold is, up to rescaling, a Wolf space.
\end{conjecture}

\subsection{Interpretation via symplectic singularities, contact vector fields, and 3-Sasakian manifolds}
Complete positive quaternionic K\"ahler manifolds are closely related to 3-Sasakian spaces and to symplectic singularities. Given a positive quaternionic K\"ahler manifold $M$, we can form its twistor space, $Z := \{(x,J) \in Q \mid J^2 = -\operatorname{Id}_{T_x M}\}$. It is canonically a projective Fano variety with a contact structure, having ample contact line bundle $L$.  The affine cone
$\operatorname{Spec} \bigoplus_{m \geq 0} H^0(Z,L^m)$  over $Z$ is a conical symplectic variety $X$ in the sense of Beauville \cite{Beauville2000} with an isolated singularity and whose contracting $\CC^\times$ action is free away from the vertex. Moreover, the Einstein metric on $M$ induces a K\"ahler--Einstein structure on $Z$ and a hyperk\"ahler  metric on $X\setminus\{0\}$, compatible with the conical structure of $X$. Conversely, any such $X$ is obtainable, by the twistor correspondence, from a complete positive quaternionic K\"ahler manifold $M$ \cite{Swann1991, LeBrun1995}.  When $M$ is a Wolf space, $X$ is the minimal nilpotent orbit closure in the corresponding simple complex Lie algebra \cite{Wolf1965, Beauville1998}. Thus the LeBrun--Salamon conjecture is equivalent to the statement that an isolated hyperk\"ahler cone $X$ with free $\CC^\times$ action away from the vertex is a minimal nilpotent orbit closure.  By Namikawa \cite{Namikawa2018}, a singular symplectic cone is  a normal nilpotent orbit closure if and only if the global functions on $X$ are generated in degree one. In the isolated setting, such a nilpotent orbit closure is the minimal nonzero one in a simple complex Lie algebra.
So the LeBrun--Salamon conjecture can alternatively be stated as saying that, for every complete positive quaternionic K\"ahler manifold, $H^0(Z,L)$ generates the algebra $\bigoplus_{n \geq 0} H^0(Z,L^{\otimes n})$ of functions on $X$. Note that $H^0(Z,L)$ identifies with the contact vector fields on $Z$ \cite[Proposition~1.1]{Beauville1998}; so the conjecture can be viewed loosely as saying that complete positive quaternionic K\"ahler manifolds always have a generating supply of contact automorphisms.

These considerations are also closely related to complete 3-Sasakian manifolds. Namely, when one of the canonical $S^1$ actions is free, the manifold is a principal $S^1$ bundle over the twistor space $Z$  of a positive quaternionic K\"ahler manifold \cite{BoyerGalicki1997}. So  the LeBrun--Salamon conjecture can also be viewed as saying that every  complete 3-Sasakian manifold with free canonical $S^1$ action is one of the homogeneous examples associated with Wolf spaces.

 \subsection{Toward the proof of the conjecture: main and previous results}
The goal of this article is to prove the LeBrun--Salamon conjecture in quaternionic dimensions up to $14$.

\begin{theorem}\label{thm:main}
Every connected complete positive quaternionic K\"ahler manifold of quaternionic dimension $n\leq14$ is, up to rescaling, a Wolf space.
\end{theorem}

The conjecture was previously established for $n\leq4$: by Hitchin \cite{Hitchin1981} and independently by Friedrich and Kurke  \cite{FriedrichKurke1982} for $n=1$, Poon and Salamon \cite{PoonSalamon1991} for $n=2$, and Buczy\'nski and Wi\'sniewski, with an appendix by Weber \cite{BuczynskiWisniewskiWeber2022}, for $n=3,4$.

Several classification results are also known under additional topological or symmetry assumptions. Fang \cite{Fang2008} proved that $M$ is a Wolf space if $\rk\Isom(M)\geq n/2+3$. Occhetta, Romano, Sol\'a Conde and Wi\'sniewski \cite{OcchettaRomanoSolaWisniewski2021} improved this bound to $\rk\Isom(M)\geq\max\{2,(n-3)/2\}$. On the topological side, the work of Galicki--Salamon and Amann \cite{GalickiSalamon1996,Amann2011} shows that $b_4(M)=1$ characterizes $\HH P^n$ among positive quaternionic K\"ahler manifolds with $3\leq n\leq6$. Kim and Lee \cite{KimLee2012} obtain the same conclusion for $n\geq5$ under the combined assumptions $b_4(M)=1$ and $\rk\Isom(M)\geq\lfloor n/2\rfloor+2$. Dancer and Swann \cite{DancerSwann1999} proved that a positive quaternionic K\"ahler manifold admitting an isometric cohomogeneity one action of a compact semisimple Lie group is a Wolf space.  Finally, LeBrun--Salamon's strong rigidity theorem \cite{LeBrunSalamon1994} shows that any counterexample to the LeBrun--Salamon conjecture has $b_2(M)=0$.

The strategy to prove Theorem~\ref{thm:main} follows the same philosophy as the proof of the cases $n=3,4$ in \cite{BuczynskiWisniewskiWeber2022}. On the one hand, one can show that a hypothetical non-Wolf quaternionic Kähler manifold of minimal dimension must have an isometry group of rank at most one: so the connected component of the identity must be isomorphic to $\{1\}, S^1, \operatorname{SU}(2)$, or $\SO(3)$.   The minimality assumption on the dimension of the counterexample allows us to get a much stronger bound on the symmetry group than in \cite{Fang2008,OcchettaRomanoSolaWisniewski2021,BuczynskiWisniewskiWeber2022}, whose results apply without this hypothesis. 

To prove this bound, we work on the associated twistor space  $Z$. If there were a torus $T$ of isometries of rank more than one, the relevant fixed component under a corank one subtorus would be positive quaternionic K\"ahler, hence a Wolf space by minimality.  Thanks to results of \cite{BuczynskiWisniewskiWeber2022}, we conclude that the fixed components $Y \subseteq Z^{T_{\mathbb{C}}}$ with extremal weights are points.  By the recognition principle \cite[Theorem 5.1]{OcchettaRomanoSolaWisniewski2021}, we conclude that $Z$ is an adjoint variety, and hence $M$ is a Wolf space.

On the other hand, we obtain a lower bound on the dimension of the isometry group of a positive quaternionic Kähler manifold. When these two bounds are incompatible, we conclude that the manifold is a Wolf space.

To obtain this lower bound, we first use Salamon's index theory to express $d-\delta_n$ as a characteristic number, where 
\[
d=\dim\Isom(M),\qquad \delta_n=\begin{cases}n+1,&n\text{ odd},\\ n+3,&n\text{ even}.\end{cases}
\]
Using the $E$--$H$ formalism and Chern--Weil theory, we write this characteristic number as the integral of a differential form expressed in terms of the scalar curvature and the quaternionic Weyl curvature.  We then prove in subsequent sections that this integral is strictly positive, yielding $d>\delta_n$, and hence $d>3$.

The positive curvature terms form a natural hierarchy.
The first level consists of positive-semidefinite matrix moments of the normalized curvature.
The second consists of even scalar orbital moments of the same curvature variable; in
Appendix~\ref{app:orbital-inclusion} the first cone is
shown to be contained in the second.  The third level consists of integrated Hodge squares,
which have a different geometric origin and need not be pointwise positive.

As an additional correctness check, we developed a partial formalization of the main result in Lean~\cite{deMouraUllrich2021}.  It verifies that every complete positive quaternionic K\"ahler manifold of quaternionic dimension $1\le n\le14$ is symmetric, assuming at most $17$ results from the literature as explicit hypotheses.  Appendix~\ref{app:lean} describes its scope and our assessment of its reliability.

\subsection{AI use}
Large language models were used extensively during the research and preparation of this article, primarily OpenAI GPT-6-Astra and GPT-5.6-Sol, with additional occasional use of Claude Opus 5, Claude Opus 5.5 and Claude Fable. They were used interactively to explore and develop proof strategies and research directions proposed by the authors, derive and check formulas, design and write computer calculations, search for relevant literature, test arguments for gaps, and prepare preliminary drafts and expository material. In particular, LLM-assisted symbolic and numerical calculations played an important role in finding the positivity decompositions used in the proof.  

The authors checked all the mathematical arguments, computations, and references, largely rewrote the article themselves, and take full responsibility for the content and its correctness. This project demonstrates the effectiveness of LLM tools in accelerating the process of learning about a research area and achieving substantial and meaningful results that can be understood, presented, and verified in a traditional manner. It was an encouraging, enabling, and educational experience for the mathematicians involved, and we hope it conveys a positive message for the future of human-AI collaborations.

\section{A minimal counterexample has symmetry group of rank at most one}
\label{sec:minimal-symmetry}

The purpose of this section is to establish the following result. Let $\Isom(M)$ denote the isometry group of a Riemannian manifold $M$. Recall, by Myers--Steenrod \cite{MyersSteenrod1939}, that this is a Lie group acting smoothly on $M$.   Recall that, if $M$ is compact, so is $\Isom(M)$: this is well-known, and a consequence of the fact that an isometry is determined by its action on the tangent space to any point, thus embedding $\Isom(M)$ into the orthonormal frame bundle of $M$.  

We also let $\CAut(Z)$ denote the group of holomorphic contact automorphisms of a holomorphic contact manifold $Z$. For us,  $Z$ will be projective, so all holomorphic automorphisms are algebraic.  

For any Lie group $G$, let $G_0 \subseteq G$ denote the connected component of the identity.  We also use the notation $\Isom_0(M) \leq \Isom(M)$  and $\CAut_0(Z) \leq \CAut(Z)$ for these connected components.

Recall also that a Riemannian manifold $(M,g)$ of dimension at least two with Ricci curvature satisfying $\operatorname{Ric}(M) \geq k g$ for some $k > 0$ is compact, by Bonnet--Myers \cite{Myers1941}.  This includes the case of a positive quaternionic K\"ahler manifold, since its metric is positive Einstein: $\operatorname{Ric}(M) = \lambda g$ for $\lambda > 0$.

Next, we recall briefly the twistor construction \cite{Salamon1982}, \cite{LeBrunSalamon1994}. Let $M$ be a connected complete positive quaternionic K\"ahler manifold. Let $Q \subset \End(TM)$ be the parallel rank-three bundle of imaginary quaternionic multiplication on the tangent bundle. Let $\pi: Z \to M$ be the twistor space of $M$, given by $Z_x := \{J \in Q_x \mid J^2 = -\operatorname{id}_{T_x M}\}$. It carries a canonical K\"ahler--Einstein metric of positive Ricci curvature, hence $K_Z^{-1}$ is positive. Hence by the Kodaira embedding theorem, $K_Z^{-1}$ is ample so $Z$ is projective and Fano, and it is connected because $M$ is.  There is a horizontal holomorphic distribution $F \subseteq TZ$ and line bundle $L := TZ/F$. Then $F$ is a contact distribution and $L$ a contact line bundle. This is essentially because the Lie bracket of vector fields induces a skew-symmetric bilinear form $\Lambda^2 F \to L$ whose restriction to horizontal lifts of vector fields on $M$ is computed by the curvature of the quaternionic bundle $Q$, so in the positive case this is a nondegenerate pairing.
The nondegenerate pairing $\Lambda^2 F \to L$ can be viewed as a section of $(\Lambda^2 F^* \otimes L)$, and its top exterior power is nowhere zero by nondegeneracy, yielding a trivialization of $\Lambda^{2n} F^* \otimes L^n$.  Thus $\det F \cong L^n$. By 
the exact sequence $0 \to F \to TZ \to L \to 0$, we conclude that $K_Z^{-1} = \det TZ \cong \det F \otimes L \cong L^{n+1}$. Since $L^{n+1} \cong K_Z^{-1}$ is positive and ample, so is $L$.

Recall that Kodaira's vanishing theorem states that, for $E$ a positive line bundle on a compact K\"ahler manifold $N$, the higher cohomology $H^{>0}(N, E \otimes K_N)$ vanishes.  In the preceding situation, we get that $H^{>0}(Z, L^j \otimes K_Z) = 0$ for all $j > 0$, that is, $H^{>0}(Z, L^{k}) = 0$ for $k > -n-1$. In particular, $H^{>0}(Z, \mathcal{O}_Z)=0$. Since $H^0(Z, \mathcal{O}_Z) = 1$, we see also that the Euler characteristic of the structure sheaf is one: $\chi(Z,\mathcal{O}_Z) = 1$. 

It is a standard fact that, since $Z$ is Fano, it is simply connected. We explain why for the convenience of the reader: let $N$ be Fano. Yau's solution of the Calabi conjecture implies that there is a K\"ahler metric whose Ricci form is a positive representative of $2 \pi c_1(N)$; however, in our situation of $N=Z$, we don't need to use this theorem since we have  its positive K\"ahler--Einstein metric.
Next, Bonnet--Myers implies that $\pi_1(N)$ is finite.  By Kodaira vanishing, $H^{>0}(N, \mathcal{O}_N)=0$ and hence $\chi(N,\mathcal{O}_N) = 1$. Now if we take the $d$-fold universal covering $\gamma: \widetilde N \to N$, 
$K_{\widetilde N}^{-1} \cong\gamma^* K_N^{-1}$ is ample since $K_N^{-1}$ is.
Thus $\widetilde N$ is also Fano.  Then $\chi(\widetilde N, \mathcal{O}_{\widetilde N})=1$. But by Hirzebruch--Riemann--Roch, $\chi(N, \mathcal{O}_N) = \int_N \operatorname{Td}(T_N)$.
 Since $\gamma^* T_N = T_{\widetilde N}$,  we deduce from this that $\int_{\widetilde N} \operatorname{Td}(T_{\widetilde N}) = d$.  But as $\widetilde N$ is also Fano, this must also equal one.  So $d=1$. We conclude that $N$ is simply connected.

 Recall that the rank, $\rk K$, of a compact Lie group is the dimension of a maximal torus contained in $K$; similarly, the rank of a reductive complex group $G$ is the dimension of a maximal complex torus contained in $G$. Note that $\rk K = \rk K_0$ and $\rk G = \rk G_0$, as every torus is contained in the connected component of the identity.
\begin{theorem}\label{M-min-symm}
Suppose that $M$ is a counterexample to the LeBrun--Salamon conjecture: it is a positive quaternionic K\"ahler manifold which is not a Wolf space. Suppose additionally that $M$ has minimal quaternionic dimension of all counterexamples.  Then $\rk \Isom(M) \leq 1$.
\end{theorem}
\begin{proof}[Proof of Theorem~\ref{M-min-symm}]
    Let $M$ be a minimal-dimension counterexample.  
    Let $K := \Isom_0(M)=\Aut_0(M,g,Q)$ with Lie algebra $\mathfrak{k}$; note that $K$ is compact. Let $G := \CAut_0(Z,F)$ with Lie algebra $\mathfrak{g}$.
    Since every element of $K$ lifts to a holomorphic contact isometry of $Z$, we have injective homomorphisms $K \to G$ and $\mathfrak{k} \to \mathfrak{g}$. We view $K$ as a subgroup of $G$. We obtain also  homomorphisms from the complexifications of $K$ and $\mathfrak{k}$: 
    \[ K_\mathbb{C} \to G, \quad \mathfrak{k}_{\mathbb{C}} \to \mathfrak{g},
    \]
    and these are isomorphisms by \cite[Corollary 4]{NittaTakeuchi1987}.
    Therefore, $G$ is reductive.
By \cite[Proposition 1.1]{Beauville1998}, the contact vector fields identify with sections of $L$:
$H^0(Z,L) \cong \mathfrak{g}$, $G$-equivariantly, with action of $G$ on $\mathfrak{g}$ given by the adjoint action.  

    Now, let $T_K \leq K$ be a maximal torus. Let $r:= \rk T_K$ be its rank.  The assertion of the theorem is that, given $M$ is a minimal counterexample, $r \leq 1$. For a contradiction, suppose $r \geq 2$. 

    Let $T := (T_K)_{\mathbb C}$ be the complexification of $T_K$. The inclusion $T_K \to K$ induces a holomorphic homomorphism $T \to K_{\mathbb C} \cong G$. The complexification of an injective homomorphism of compact Lie groups is again injective: this follows from polar decomposition; see \cite[Ch.~III, \S 8, Proposition 8.3, Exercises~8.8(1),(4)]{BrockerTomDieck1985}. We may therefore regard $T$ as a subgroup of $G$.  Since it is a complex torus of rank $r=\rk G$, it is a maximal torus of $G$ containing $T_K$.

    Let $X^*(T)$ denote the character lattice of $T$, which is isomorphic to $\mathbb{Z}^{r}$.  We will also refer to its elements as weights when they occur in $T$-representations, including the fibers of $L$.
    We will define some polytopes in it by taking convex hulls of certain sets of weights. 
    
    We consider the fixed locus $Z^T$. It cannot be empty, since $Z$ is projective: every positive-dimensional orbit of $T$ has a nontrivial boundary consisting of orbits of smaller dimension, hence there must be $T$-orbits of dimension zero, hence fixed points. For each $y \in Z^T$, $T$ preserves the line $L_y$, and acts on it by a weight. On a connected component $Y \subset Z^T$ of the fixed locus, this weight must be constant (since $X^*(T)$ is discrete). Call this weight $\mu(Y)$. Let $H^0(Z,L)_{\mu(Y)} \subset H^0(Z,L)$ be the subspace of sections of $Z$ on which $T$ acts via the weight $\mu(Y)$. Restricting  sections
    gives a $T$-equivariant linear map
    \begin{equation}\label{rest-twistor}
    H^0(Z,L)_{\mu(Y)} \to H^0(Y, L).
    \end{equation}
     Note that, since $T$ is a torus, $H^0(Z,L) = \bigoplus_{\lambda \in X^*(T)} H^0(Z,L)_\lambda$; for every $\lambda$, there is similarly a restriction map $H^0(Z,L)_\lambda \to H^0(Y,L)$, but this is zero for $\lambda \neq \mu(Y)$, since $T$ acts on $H^0(Y,L)$ by the weight $\mu(Y)$, and the restriction map commutes with the action of $T$.

     Now, suppose that $\mu(Y) \neq 0$. Let $S := \ker(\mu(Y)|_{T_K})_0 \leq T_K$. 
    This is a subtorus of $T_K$ of corank one. By our assumption $r \geq 2$, $S$ is nontrivial.  Now $\pi(Y)$ is connected and $\pi(Y) \subset M^S$.  Let $N \subset M^S$ be the component containing $\pi(Y)$. We claim that $N$ is a quaternionic K\"ahler submanifold of $M$.  To see this, first observe that, for every $y=(x,J) \in Y$, $S$ acts trivially on $L|_y$, by definition.  By construction $L|_y$ identifies with the tangent space $T_y Z_x$ to the twistor fiber $Z_x$.  Since $S$ fixes $y$ and acts trivially on this tangent line, its induced action on $Z_x$ fixes $J$ and has trivial isotropy representation on $T_J Z_x$. Thus $S$ acts trivially on $Q_x$.  Since the possible weights of $S$ on $Q_w$ for $w \in N$ form a  discrete set, $N$ is connected, and $x \in N$, we deduce that all these weights are zero and $S$ acts trivially on $Q_w$, for all $w \in N$.  By \cite[Proposition 5.1 and its proof]{GoertschesLoiudice2025}, $N$ is therefore positive quaternionic K\"ahler.  For the convenience of the reader we recall the reason:    $Q$ preserves $TN$ and defines an almost quaternionic structure. But $N$ is also totally geodesic, being locally defined as the fixed locus for $S$.  So the ambient parallel bundle $Q$ on $M$ restricts to a parallel rank-three subbundle of $\End(TN)$. Thus $N$ is quaternionic K\"ahler by definition if $\dim N > 4$.  For dimension four, one has to check that $N$ is Einstein with positive scalar curvature and that its twistor space has an integrable almost complex structure: the former uses the decomposition of curvature into a scalar part and a quaternionic Weyl part, whereas the latter uses that the twistor bundle and the almost complex structure on the twistor space are obtained by  restriction from $M$, where the almost complex structure on $Z$ is integrable; see \emph{op.~cit.~} and references therein for details. Since $N$ is closed in $M$, which is compact, $N$ is also compact and hence complete.

    If $N$ is a point, then $Y \subset \pi^{-1}(N) \cong S^2$, which has a nontrivial action of $T$ since $\mu(Y) \neq 0$, and hence it has exactly two fixed points. Thus in this case $Y$ is also a point.

    If $N$ is not a point, the faithfulness of the action of $T_K$ on $M$ shows that $\dim N < \dim M$. Since $N$ is quaternionic K\"ahler and $M$ is a minimal counterexample to the LeBrun--Salamon conjecture, it follows that $N$ is a Wolf space.  Its twistor space is the homogeneous contact space associated with the highest root \cite[Theorem 6.1]{Wolf1965}, equivalently the adjoint variety \cite[Proposition 2.6]{Beauville1998}, with embedding given by $L|_{Z_N}$, where $Z_N := \pi^{-1}(N) \subset Z$. This line bundle is very ample, as is its restriction $L|_Y$ to $Y \subset Z_N$. 
    As a result, we obtain the inequality:
    \begin{equation}
    \label{e:h0y-bound}
    \dim H^0(Y, L)  \geq \dim Y + 1.
    \end{equation}
    In summary, either $Y$ is a point, or else $\dim H^0(Y,L) \geq \dim Y + 1 \geq 2$.

To proceed, we need to know the following basic fact about a hypothetical counterexample to the LeBrun--Salamon conjecture:
\begin{lemma}
$\Pic(Z) = \mathbb{Z} \cdot [L] \cong \mathbb{Z}$.
\end{lemma}
\begin{proof}
    As noted in \cite[Remark 4.1]{OcchettaRomanoSolaWisniewski2021}, \cite[Proposition 2.13]{KebekusPeternellSommeseWisniewski2000} implies that every contact Fano manifold has Picard group isomorphic to \(\mathbb Z\), unless it is isomorphic to a type $A$ adjoint variety, 
    which is the twistor space of a Wolf space. We can exclude this case by \cite[Corollary 3.4]{LeBrun1995}.  However, we 
 give a direct argument in the present twistor setting.
    
    Since $M$ is not a Wolf space, strong rigidity \cite{LeBrunSalamon1994} implies that $b_2(M)=0$. Since $Z$ is an $S^2$ bundle over $M$, the Leray--Serre spectral sequence  implies that $b_2(Z) \in \{0,1\}$. Since the class $c_1(L)$ of the contact line bundle restricts to the anticanonical class $c_1(\mathcal{O}(2)) \neq 0$ on the fibers $Z_x \cong \mathbb{CP}^1$, we must have $b_2(Z) = 1$. Since $Z$ is Fano, it is simply connected.  We deduce from the universal coefficient theorem that $H^2(Z,\mathbb{Z}) \cong \mathbb{Z}$:
    \[
    0 = \Ext(H_1(Z,\mathbb{Z}),\mathbb{Z}) \to H^2(Z,\mathbb{Z}) \to \Hom(H_2(Z,\mathbb{Z}),\mathbb{Z}) \cong \mathbb{Z}.
    \]
    We now apply the long exact sequence associated to the exponential sequence $0 \to \mathbb{Z} \to \mathcal{O} \to \mathcal{O}^\times \to 0$, giving by Kodaira vanishing the exact sequence
    \[
    0 =H^1(Z,\mathcal{O}_Z) \to H^1(Z,\mathcal{O}^\times_Z) \to H^2(Z,\mathbb{Z}) \to H^2(Z,\mathcal{O}_Z)=0.
    \]
So $H^1(Z,\mathcal{O}_Z^\times) \cong \mathbb{Z}$.
    But this cohomology group is the analytic Picard group, which we conclude is $\mathbb{Z}$. Since $Z$ is projective, by GAGA, the analytic and algebraic Picard groups coincide (this can also be deduced from Chow's theorem via the divisorial viewpoint). Thus $\Pic(Z) \cong \mathbb{Z}$.  
    
    By \cite[Remark 4.1]{OcchettaRomanoSolaWisniewski2021},
$\Pic(Z) = \mathbb{Z} \cdot [L]$, unless  $Z$ is  a projective space; we give an argument for the reader's benefit. First, $\Pic(Z) = \langle [E] \rangle$ for some line bundle $E$, and $L \cong E^a$ for some $a \in \mathbb{Z}$. We have $a \neq 0$ since $L$ is positive, and we can assume $a \geq 1$ up to replacing $E$ with $E^{-1}$, so that $E$ is a positive primitive generator of $\Pic(Z)$. We obtain also that $-[K_Z] = (n+1)[L] = a(n+1)[E]$.  The number $a(n+1)$ is called the Fano index of $Z$.
    
    By \cite{KobayashiOchiai1973}, for a Fano $m$-fold, the Fano index is at most $m+1$, with equality if and only if the variety is isomorphic to a projective space. So $a(n+1) \leq 2n+2$, that is, $a \leq 2$, with equality if and only if $Z \cong \mathbb{P}^{2n+1}$. 
    
    In the latter case, $M \cong \mathbb{H}\mathbb{P}^n$ by \cite[Corollary 3.3]{LeBrun1995}.  Thus for $M$ not to be a Wolf space, we must have $a=1$. So $\Pic(Z) = \langle L \rangle$.
\end{proof}
    
    Next, let $\Delta := \operatorname{Conv}(\{\mu(Y): Y \subseteq Z^T \text{ is a component}\}) \subseteq X^*(T) \otimes_{\mathbb{Z}} \mathbb{R} \cong \mathfrak{t}^*$ for $\mathfrak{t} = \Lie T_K$.  
    The vertices of $\Delta$ are called \emph{extremal weights}.  By \cite[Lemma 3.6]{BuczynskiWisniewskiWeber2022}  (see also \cite[Lemma 5.3]{OcchettaRomanoSolaWisniewski2021}), if $\mu(Y)$ is extremal, then \eqref{rest-twistor}  is an isomorphism.   Moreover, by \cite[Lemma 3.2, Corollary 4.2]{BuczynskiWisniewskiWeber2022}, $\mu(Y) \neq 0$ for every extremal weight.

 Now, if $Y$ is a point, then $H^0(Y,L) \cong \CC \neq 0$. If not, then \eqref{e:h0y-bound} implies that $H^0(Y,L) \neq 0$. Either way, by surjectivity of \eqref{rest-twistor}, we must  have $H^0(Z,L)_{\mu(Y)} \neq 0$. Since $H^0(Z,L) \cong \mathfrak{g}$ with $G$ acting by the adjoint representation, and $\mu(Y) \neq 0$, we conclude that $\mu(Y)$ is a root of $G$ (as was also observed in \cite[Lemma 5.3]{OcchettaRomanoSolaWisniewski2021}). Then $\dim H^0(Z,L)_{\mu(Y)} = 1$, and by surjectivity of \eqref{rest-twistor}, $\dim H^0(Y,L) \leq 1$. As a result, \eqref{e:h0y-bound} is impossible. So $Y$ must be a point.

From the preceding, we conclude that all extremal weight components $Y \subset Z^T$ are points.   Thus by \cite[Theorem 5.1]{OcchettaRomanoSolaWisniewski2021}, $Z$ is an adjoint variety, hence the twistor space of a Wolf space. By the uniqueness of the twistor space \cite[Corollary 3.4]{LeBrun1995}, $M$ must be a Wolf space, which supplies the contradiction. \qedhere

\end{proof}

\begin{remark}
The argument above applies in all dimensions. For $n\leq4$, however,
\cite[Theorem~5.1]{OcchettaRomanoSolaWisniewski2021} uses the known cases of the
LeBrun--Salamon conjecture proved in \cite{BuczynskiWisniewskiWeber2022}, so our proof in this
range is not independent of that work. The content of this result is the symmetry bound for
minimal counterexamples in dimensions $n\geq5$.
\end{remark}

\begin{corollary}\label{cor:rank-one-isometry}
For $M$ as in Theorem~\ref{M-min-symm}, $\Isom_0(M)$ is isomorphic to either $\{1\}, S^1$, $\operatorname{SU}(2)$, or $\operatorname{SO}(3)$.
\end{corollary}
\begin{proof} As we observed at the beginning of the section, $M$ is compact, and as a consequence, so is $\Isom(M)$.  Then the proof reduces to the classification of connected compact Lie groups of rank at most one. 
\end{proof}
In particular, the dimension of $\Isom_0(M)$, hence also of $\Isom(M)$, can be at most $3$ for a minimal counterexample $M$.

\section{Quaternionic characteristic classes and Salamon's Hilbert polynomial}
\label{sec:index-moments}

In this section we relate the dimension $d=\dim \Isom_0(M)$ of the isometry group of a positive quaternionic--K\"ahler manifold $M^{4n}$ to characteristic classes of $M$. This relation will be used in the next section to obtain a lower bound for $d$. In the first part of this section we recall the quaternionic geometric and characteristic-class formalism that will be used later in the positivity arguments. In the second part we use Salamon's Hilbert polynomial to relate the relevant characteristic number with $d$.

\subsection{Quaternionic--K\"ahler geometry and characteristic classes}

The purpose of the first two parts of the subsection is to fix the notation and explain how the local $E$--$H$ formalism gives rise to global characteristic classes. In the third part, we then construct directly a second connection on the local bundles $E_i\oplus H_i$ whose curvature is the quaternionic Weyl curvature. We follow the standard descriptions in \cite{Besse1987,Kraines1966,Salamon1982}; for the explicit $E$--$H$ curvature decomposition used in the direct construction we refer to \cite{KramerSemmelmannWeingart1998}.

\subsubsection{The Kraines form and the \texorpdfstring{$\mathfrak{sp}(1)$}{sp(1)}-curvature}

Recall that $Q$ is the rank-three subbundle $Q\subset\End(TM)$ preserved by the Levi--Civita connection and locally spanned by orthogonal almost-complex structures $I,J,K$ satisfying the quaternionic relations
$$I^2=J^2=K^2=-\operatorname{Id},\qquad IJ=K=-JI.$$
Although $Q$ is globally defined, the endomorphisms $I,J,K$ are in general only local choices. At every $x\in M$, the fiber $Q_x$ is modeled on $\operatorname{Im}\HH$, and changing a local quaternionic frame acts on $Q_x$ through the standard double covering $\Sp(1)\to\SO(3)$.

For a local quaternionic frame $(I,J,K)$, let
$$\omega_I(X,Y)=g(X,IY),\qquad \omega_J(X,Y)=g(X,JY),\qquad \omega_K(X,Y)=g(X,KY)$$
be the corresponding fundamental $2$-forms. Since the triple $(\omega_I,\omega_J,\omega_K)$ changes by the same $\SO(3)$-rotation as $(I,J,K)$, the $4$-form
$$\Omega:=\omega_I\wedge\omega_I+\omega_J\wedge\omega_J+\omega_K\wedge\omega_K$$
is independent of the local quaternionic frame. This is the Kraines form \cite{Kraines1966,Salamon1982}. Since $Q$ is parallel and the expression above is $\SO(3)$-invariant, $\Omega$ is parallel and in particular closed.

The quaternionic structure determines a canonical orientation. Namely, if $e_1,\ldots,e_n$ is a quaternionic orthonormal basis, we declare
$$\{e_1,Ie_1,Je_1,Ke_1,\ldots,e_n,Ie_n,Je_n,Ke_n\}$$
to be positively oriented. With this convention,
\begin{equation}\label{eq:kraines-form}
\Omega^n=(2n+1)!\,\vol_g.\end{equation}

The same local forms describe the $\mathfrak{sp}(1)$-part of the Levi--Civita curvature. The curvature takes values in $\mathfrak{sp}(n)\oplus\mathfrak{sp}(1)$ and decomposes into the quaternionic Weyl curvature and a scalar-curvature term; see \cite[Theorem 3.1]{Salamon1982}. The quaternionic Weyl curvature has zero $\mathfrak{sp}(1)$-projection, so the $\mathfrak{sp}(1)$-part is completely determined by $\kappa$. If $J_1,J_2,J_3$ are the generators of $\mathfrak{sp}(1)$ acting on the tangent representation as $I,J,K$, with trace taken in the standard complex two-dimensional representation, so that $\tr_H(J_\alpha J_\beta)=-2\delta_{\alpha\beta}$, then, with our conventions,
$$R^{\mathfrak{sp}(1)}=\frac{\kappa}{8n(n+2)}\bigl(\omega_IJ_1+\omega_JJ_2+\omega_KJ_3\bigr).$$

\subsubsection{The \texorpdfstring{$E$--$H$}{E--H} formalism and characteristic classes}\label{subsubsec:EH-formalism}

We now recall the usual $E$--$H$ formalism for quaternionic--K\"ahler manifolds \cite{Salamon1982}. The complexified tangent representation of $G=\Sp(n)\Sp(1)$ is the tensor product of the standard complex representations
$$E\cong\CC^{2n},\qquad H\cong\CC^2.$$
Thus, if $P\to M$ denotes the principal $G$-bundle of quaternionic orthonormal frames, then globally
$$T_\CC M\cong P\times_G(E\otimes H).$$
The tensor product $E\otimes H$ is a representation of $G$, because the central element $(-1,-1)\in\Sp(n)\times\Sp(1)$ acts trivially on it. By contrast, $E$ and $H$ separately do not in general descend to global vector bundles on $M$. 

Choose a good cover $\{U_i\}$ of $M$. Since each $U_i$ is contractible, $P|_{U_i}$ lifts to a principal $\Sp(n)\times\Sp(1)$-bundle, and we obtain complex vector bundles $E_i\to U_i$ and $H_i\to U_i$ such that
$$T_\CC M|_{U_i}\cong E_i\otimes H_i.$$
The Levi--Civita connection on the principal $G$-bundle $P$ lifts to the local $\Sp(n)\times\Sp(1)$-bundle. After choosing local trivializations, writing its connection form as a pair
$$\omega_i=(\omega_i^E,\omega_i^H),\qquad \omega_i^E\in\Omega^1(U_i,\mathfrak{sp}(n)),\quad \omega_i^H\in\Omega^1(U_i,\mathfrak{sp}(1)),$$
we obtain induced connections $\nabla^{E_i}$ and $\nabla^{H_i}$ on $E_i$ and $H_i$. Under the tensor-product representation $E_i\otimes H_i$, the complexified Levi--Civita connection is therefore
$$\nabla^{LC}=\nabla^{E_i}\otimes 1+1\otimes\nabla^{H_i},$$
with curvature
$$F_i^{TM}=F_i^E\otimes\operatorname{Id}_{H_i}+\operatorname{Id}_{E_i}\otimes F_i^H.$$
Thus the two curvature components $F_i^E$ and $F_i^H$ are simply the two components of the same lifted Levi--Civita curvature, viewed through the two standard representations. Set
$$X_i^E:=\frac{\sqrt{-1}}{2\pi}F_i^E,\qquad X_i^H:=\frac{\sqrt{-1}}{2\pi}F_i^H.$$
Since $F_i^E$ and $F_i^H$ take values in $\mathfrak{sp}(n)$ and $\mathfrak{sp}(1)$ respectively, their formal roots occur in opposite pairs. We write the formal roots of $E_i$ as
\begin{equation}\label{eq:formal-roots-E}\pm x_{1,i},\ldots,\pm x_{n,i}\end{equation}
and those of $H_i$ as \begin{equation}\label{eq:formal-roots-H}  \pm h_i. \end{equation} The symbols $x_{a,i}$ and $h_i$ are formal Chern roots; they need not be actual $2$-forms. The formal roots of $T_\CC M|_{U_i}\cong E_i\otimes H_i$ are therefore
$$\pm x_{a,i}\pm h_i,\qquad 1\leq a\leq n.$$

We first show how to obtain a global closed 4-form from the formal roots $\pm h_i$. On an overlap $U_i\cap U_j$, choose a lift $(a_{ij},q_{ij})$ of the transition function to $\Sp(n)\times\Sp(1)$. The lift is unique up to multiplication by $(-1,-1)$, while the normalized $H$-curvatures transform by conjugation,
$$X_j^H=q_{ij}^{-1}X_i^Hq_{ij}.$$
Since changing $q_{ij}$ to $-q_{ij}$ has no effect on this formula, the conjugation-invariant forms $\frac12\tr\bigl((X_i^H)^2\bigr)$ agree on overlaps. Moreover, they are Chern--Weil forms and therefore they are closed. Since $\pm h_i$ are the formal roots of $H_i$, then on each $U_i$
$$h_i^2=\frac12\tr\bigl((X_i^H)^2\bigr).$$
Thus these local $4$-forms patch to a global closed $4$-form, which we denote by 
$$u:=h^2.$$
We use the same notation for a global closed characteristic form and its de Rham class when the meaning is clear. Since $Q_\CC\cong\Sym^2H_i$ locally,
\begin{equation}\label{eq:u-p1Q}4u=p_1(Q)=\left(\frac{\kappa}{8\pi n(n+2)}\right)^2\Omega,\end{equation} see \cite[\S 2]{SemmelmannWeingart2004}. Thus $u$ is rational and is represented by a positive multiple of the Kraines form. 

A similar discussion applies to the formal roots $\pm x_{a,i}$. Put $q_{0,i}=n$ and \begin{equation}\label{eq:q-j}q_{j,i}=\sum_{a=1}^n x_{a,i}^{2j}\end{equation} for $j\geq1$. If $\tau_j(TM)$ is the $j$-th Pontryagin power sum, then
$$\tau_j(TM)_{|U_i}=\sum_{a=1}^n\bigl((x_{a,i}+h)^{2j}+(x_{a,i}-h)^{2j}\bigr)
=2\sum_{s=0}^j\binom{2j}{2s}u^{j-s}q_{s,i}.$$
This triangular identity expresses each $q_{j,i}$ as a universal rational polynomial in the Pontryagin classes of $TM$ and $u$. Hence symmetric polynomials in the $x_a^2$ define global rational characteristic classes and we can consider global rational characteristic classes $q_j=\sum_{a=1}^n x_{a}^{2j}$.

\subsubsection{A Weyl-curvature connection and the forms \texorpdfstring{$Y$ and $z_j$}{Y and z j}}\label{subsubsec:weyl-Y-z}

We now consider the local direct sums $E_i\oplus H_i$, of complex rank $r:=2n+2$, associated to the same lifted principal bundles as $E_i\otimes H_i$, and equip them with $\nabla_i^0=\nabla^{E_i}\oplus\nabla^{H_i}$. Their formal Chern polynomial is
\begin{equation}\label{eq:chern-polynomial-EH}
C(t)=(1-ut^2)\prod_{a=1}^n(1-x_a^2t^2)=1+c_2t^2+c_4t^4+\cdots.
\end{equation}

By the preceding discussion its coefficients define global rational classes $c_{2j}\in H^{4j}(M;\mathbb Q)$. Our goal is to construct representatives involving only the quaternionic Weyl curvature.

Let $\sigma_E$ and $\sigma_H$ be the nondegenerate skew-symmetric complex bilinear forms preserved by the standard representations of $\Sp(n)$ and $\Sp(1)$. We normalize the identification $T_\CC M|_{U_i}\cong E_i\otimes H_i$ so that the complex-bilinear extension of the metric satisfies
$$g_\CC(e_1\otimes h_1,e_2\otimes h_2)=\sigma_E(e_1,e_2)\sigma_H(h_1,h_2).$$
These forms give the equivariant identifications $E_i\cong E_i^*$ and $H_i\cong H_i^*$. For $\xi=e\otimes h\in E_i\otimes H_i$ define
$$A_\xi\in\End(E_i\oplus H_i),\qquad A_{e\otimes h}(e',h')=\bigl(\sigma_H(h,h')e,\ \sigma_E(e,e')h\bigr).$$
Extending complex-linearly defines $A_\xi$ for arbitrary $\xi\in E_i\otimes H_i$. The map $\xi\mapsto A_\xi$ is $\Sp(n)\times\Sp(1)$-equivariant. Thus, for $X\in TM$, setting $\Phi_i(X):=A_{X_\CC}$ defines
$$\Phi_i\in\Omega^1\bigl(U_i,\End(E_i\oplus H_i)\bigr),$$
where $X_\CC$ is viewed in $E_i\otimes H_i$ through the above identification. All subsequent identities may be checked after complex-linear extension to $T_\CC M$.

Let $h_E,h_H$ be the standard Hermitian metrics, compatible with the symplectic forms and quaternionic structures $J_E,J_H$. Writing $A_\xi=\left(\begin{smallmatrix}0&B_\xi\\C_\xi&0\end{smallmatrix}\right)$, one checks that $B_\xi^*=-C_{(J_E\otimes J_H)\xi}$. Thus $A_X$ is skew-Hermitian for real $X$, and it also preserves $\sigma_E\oplus\sigma_H$; see \cite[\S2.1]{KramerSemmelmannWeingart1998}.

For $e_1,e_2\in E_i$ and $h_1,h_2\in H_i$, write $e_1e_2\in\mathfrak{sp}(E_i)$ and $h_1h_2\in\mathfrak{sp}(H_i)$ for the symmetric products
$$(e_1e_2)e=\sigma_E(e_1,e)e_2+\sigma_E(e_2,e)e_1,$$
and similarly on $H_i$. Here $\mathfrak{sp}(E_i)$ and $\mathfrak{sp}(H_i)$ denote the complex symplectic Lie algebras determined by $\sigma_E$ and $\sigma_H$. A direct calculation gives, for $\xi=e_1\otimes h_1$ and $\eta=e_2\otimes h_2$,
$$[A_\xi,A_\eta]=
\begin{pmatrix}
\sigma_H(h_1,h_2)e_1e_2&0\\
0&\sigma_E(e_1,e_2)h_1h_2
\end{pmatrix}.$$
These two diagonal terms are precisely the model curvature tensors $R^E$ and $R^H$ in the standard $E$--$H$ description of quaternionic--K\"ahler curvature; see \cite[\S2.2]{KramerSemmelmannWeingart1998}.

The equivariant map $\xi\mapsto A_\xi$ is parallel. Torsion-freeness therefore gives $d^{\nabla_i^0}\Phi_i=0$.

Set
$$a:=\frac{\kappa}{8n(n+2)}>0.$$
The standard curvature decomposition of a quaternionic--K\"ahler metric is
$$R=-a(R^H+R^E)+R^{\mathrm{hyper}},$$
where $R^{\mathrm{hyper}}$ is the quaternionic Weyl curvature; see \cite[Lemma~2.4]{KramerSemmelmannWeingart1998} and \cite[Theorem 3.1]{Salamon1982}. The tensors $R^H$ and $R^E$ are respectively the $\mathfrak{sp}(1)$- and $\mathfrak{sp}(n)$-model curvatures, whereas $R^{\mathrm{hyper}}$ acts only through the $E_i$-factor. Thus the curvature of $\nabla_i^0$ on $E_i\oplus H_i$ is
$$F_i^0(\xi,\eta)
=-a[A_\xi,A_\eta]
+\begin{pmatrix}W_i(\xi,\eta)&0\\0&0\end{pmatrix},$$
where $W_i$ denotes the $\mathfrak{sp}(n)$-valued quaternionic Weyl curvature.

We now modify the connection by the tensorial one-form $\Phi_i$.

\begin{lemma}\label{lem:weyl-connection}
The connection
$$D_i:=\nabla_i^0+\sqrt a\,\Phi_i$$
on $E_i\oplus H_i$ preserves $h_E\oplus h_H$ and $\sigma_E\oplus\sigma_H$ and has curvature
$$F_i^D=\begin{pmatrix}W_i&0\\0&0\end{pmatrix}.$$
\end{lemma}

\begin{proof}
The compatibility statements follow from those of $\nabla_i^0$ and $\Phi_i$. For a connection changed by an endomorphism-valued $1$-form,
$$F_i^D=F_i^0+\sqrt a\,d^{\nabla_i^0}\Phi_i+a\,\Phi_i\wedge\Phi_i.$$
The middle term vanishes, while
$$(\Phi_i\wedge\Phi_i)(\xi,\eta)=[A_\xi,A_\eta].$$
Substituting the preceding expression for $F_i^0$ gives the claimed cancellation of the scalar model curvature.
\end{proof}

The pair symmetry of quaternionic Weyl curvature places its two-form factor in $\mathfrak{sp}(n)\subset\Lambda^2T^*M$. Consequently
$$\frac{\sqrt{-1}}{2\pi}F_i^D\in\mathfrak{sp}(n)\otimes_{\RR}\Herm(E_i\oplus H_i);$$
see \cite[Lemma~2.4]{KramerSemmelmannWeingart1998}.

\begin{remark}
The scalar-curvature cancellation above is the local $E$--$H$ counterpart of the standard quaternionic local twistor connection; compare \cite[\S3, equations~(14)--(16)]{GoverSlovak1999}.
\end{remark}

Set
\begin{equation}\label{eq:def-Y-z}
X_i:=\frac{\sqrt{-1}}{2\pi}F_i^D,\qquad Y_i:=-X_i^2,\qquad z_j^{(i)}:=\tr(Y_i^j), 
\end{equation}
where products combine wedge product and endomorphism composition. Equivariance gives $F_j^D=g_{ij}^{-1}F_i^Dg_{ij}$ on overlaps, and hence
\begin{equation}\label{eq:X-overlap}X_j=g_{ij}^{-1}X_i g_{ij}.\end{equation} The central ambiguity in $g_{ij}$ acts as $-\operatorname{Id}$ and hence trivially by conjugation, so the invariant forms patch globally. 

The connections $D_i$ and $\nabla_i^0$ are compatible connections on the same projective bundle. By Chern--Weil independence, their invariant forms represent the same global characteristic classes; see \cite[Chapter~XII, \S1]{KobayashiNomizu1969}. Thus the coefficients of $\det(\operatorname{Id}+tX_i)$ represent the same classes $c_{2j}$ defined in \eqref{eq:chern-polynomial-EH}. We use $c_{2j}$ also for these representatives, and write $z_j=\tr(Y^j)$. The odd Chern forms vanish because the non-zero roots of $X_i$ appear in opposite pairs.

Using Newton's identities, we can explicitly describe $z_j$ in terms of $c_2,\dots, c_{2j}$. As an example, the first few terms give 
$$z_1=2c_2,\qquad z_2=2c_2^2-4c_4,\qquad z_3=2c_2^3-6c_2c_4+6c_6.$$
Equivalently, comparing with the Levi--Civita direct-sum representative,
\begin{equation}\label{eq:cohomologous-z-q-u}\left[\frac{z_j}{2}\right]=(-1)^j\bigl(q_j+[u]^j\bigr)\quad\text{in }H^{4j}(M;\RR), \end{equation}
where $q_j$ is the class defined in \eqref{eq:q-j}. 

\subsection{Salamon's Hilbert polynomial}

Recall that $Z$ denotes the twistor space of $M$ and $L$ denotes its contact line bundle, which is an ample holomorphic line bundle on $Z$; see \cite[Corollary 6.2]{Salamon1982}. We consider its Hilbert polynomial
$$P(r):=\chi(Z,L^r)=\sum_{j=0}^{2n+1}(-1)^jh^j(Z,L^r).$$
By the Hirzebruch--Riemann--Roch theorem,
$$P(r)=\left\langle e^{rc_1(L)}\td(TZ),[Z]\right\rangle.$$
Since $Z$ has complex dimension $2n+1$ and the first Chern class $c_1(L)\in H^{1,1}(Z,\RR)$ is positive, the leading term of $P(r)$ is $$\frac{\langle c_1(L)^{2n+1},[Z]\rangle}{(2n+1)!}r^{2n+1},$$ whose coefficient is positive. Hence $P(r)$ is a polynomial of degree $2n+1$. We shall use the following special values.

\begin{samepage}
\begin{proposition}\label{prop:salamon-hilbert}\cite[\S2]{SemmelmannWeingart2004}
The polynomial $P(r)$ satisfies
\begin{itemize}
    \item[1.] $P(0)=1$,
    \item[2.] $P(-1)=P(-2)=\cdots=P(-n)=0$,
    \item[3.] $P(1)=\dim\Isom_0(M)=d$.
\end{itemize}
\end{proposition}
\end{samepage}

The first two assertions follow from the relation $-K_Z\cong L^{\otimes(n+1)}$ and Kodaira vanishing. Indeed, for $r\geq -n$ we have $L^{\otimes r}\cong K_Z\otimes L^{\otimes(n+1+r)}$. Since $L$ is ample, Kodaira vanishing gives $H^j(Z,L^{\otimes r})=0$ for $j>0$ and $r\geq -n$. Consequently, $P(0)=h^0(Z,\mathcal O_Z)=1$. For $-n\leq r\leq -1$, the bundle $L^{\otimes r}$ has no nonzero holomorphic sections, so $P(-1)=P(-2)=\cdots=P(-n)=0$. Finally, by \cite[Proposition~1.1]{Beauville1998}, $H^0(Z,L)\cong\mathfrak g$, where $\mathfrak g$ is the Lie algebra of holomorphic contact automorphisms of $(Z,L)$. By \cite[Corollary~4]{NittaTakeuchi1987} and \cite[Proposition 1.1]{Beauville1998}, $\mathfrak g\cong\mathfrak k_{\mathbb C}$, where $\mathfrak k$ is the Lie algebra of $K=\operatorname{Isom}_0(M)$. Hence $P(1)=h^0(Z,L)=\dim_{\mathbb C}\mathfrak g=\dim_{\mathbb C}\mathfrak k_{\mathbb C}=\dim_{\mathbb R}K=d$.

Integrating the Hirzebruch--Riemann--Roch expression along the fibers of the twistor fibration $\pi:Z\to M$ gives, for $n+2r\geq0$,
$$P(r)=\left\langle\Ahat(TM)\ch(\Sym^{n+2r}H),[M]\right\rangle;$$
see \cite[Lemma 7.2]{Salamon1982} and \cite[\S 2, (2.2)]{SemmelmannWeingart2004}. Although $H$ need not exist globally, this characteristic expression is globally defined. Indeed, if the formal roots of $H$ are $\pm h$, then
\begin{equation}\label{eq:ch-SymH}\ch(\Sym^qH)=\sum_{j=0}^qe^{(q-2j)h}.\end{equation}
The weights occur in opposite pairs, so every cohomological component is an even polynomial in $h$, and hence a polynomial in the globally defined class $u=h^2$.

We define the virtual $\Sp(1)$-representation with character
$$\chi(V_n)(t)=\begin{cases}(t-t^{-1})^{n+2},&\text{if $n$ is even},\\(t-t^{-1})^{n+1}(t+t^{-1}),&\text{if $n$ is odd},\end{cases}$$
and set
$$\delta_n=\begin{cases}n+3,&\text{if $n$ is even},\\n+1,&\text{if $n$ is odd}.\end{cases}$$
If $\chi_q(t)=t^q+t^{q-2}+\cdots+t^{-q}$ denotes the character of $\Sym^qH$, then
$$\chi(V_n)=\chi_{n+2}-\delta_n\chi_n+\sum_{j=1}^{\lfloor n/2\rfloor}a_{n,j}\chi_{n-2j}$$
for certain integers $a_{n,j}$. Indeed, the defining Laurent polynomial is invariant under $t\mapsto t^{-1}$ and has integral coefficients, so it is an integral combination of the characters $\chi_q$. Its leading coefficient is $1$, while the coefficient of $t^n$ is $-(n+2)$ for even $n$ and $-n$ for odd $n$. Subtracting $\chi_{n+2}$ therefore gives the coefficient $-\delta_n$ of $\chi_n$; the remaining weights give the displayed lower terms. We set
$$K_n:=\bigl[\Ahat(TM)\ch(V_n)\bigr]_{4n}.$$

\begin{proposition}\label{prop:virtual-index}
With $d=\dim\Isom_0(M)$, we have
$$\int_MK_n=d-\delta_n.$$
\end{proposition}

\begin{proof}
By additivity of the Chern character and the formula for $P(r)$,
$$\int_MK_n=P(1)-\delta_nP(0)+\sum_{j=1}^{\lfloor n/2\rfloor}a_{n,j}P(-j).$$
By Proposition~\ref{prop:salamon-hilbert}, the last sum vanishes, while $P(0)=1$ and $P(1)=d$. Hence $\int_M K_n=d-\delta_n$.
\end{proof}

\begin{remark}\label{rem:Kn-cohomologous}
In the next section we use the representative cohomologous to $K_n$ obtained from the Weyl-curvature forms $z_j$ introduced in Subsection~\ref{subsubsec:weyl-Y-z}. Since $\int_MK_n$ only depends on the cohomology class of $K_n$ and not the representative, we will also denote the representative obtained from the Weyl-curvature forms by $K_n$.
\end{remark}

\begin{remark}
There is also an index-theoretic interpretation of the preceding proposition. Although $H$ itself need not be a global vector bundle, the spin representation tensored with $\Sym^{n+2r}H$ descends through $\Sp(n)\times\Sp(1)\to\Sp(n)\Sp(1)$ and defines a global twisted Dirac operator $D_{n+2r}$; see \cite[\S 2--3]{SemmelmannWeingart2004}. The character expansion of $V_n$ therefore determines a virtual combination $D_{V_n}$ of these operators. By additivity of the index and the Atiyah--Singer theorem \cite{AtiyahSinger1968},
$$\ind(D_{V_n})=\int_M\Ahat(TM)\ch(V_n)=\int_MK_n=d-\delta_n.$$
\end{remark}

\section{Curvature positivity and positive decompositions}\label{sec:positivity}

By Proposition~\ref{prop:virtual-index},
\[
       \int_MK_n=d-\delta_n.
\]
For a minimal counterexample, Section~\ref{sec:minimal-symmetry} gives \(d\le3\), whereas \(\delta_n>3\) in every dimension considered below.  It therefore suffices to prove
\[
        \int_MK_n>0.
\]

\subsection{The curvature polynomial \(K_n\)}
\label{subsec:curvature-polynomial}

We first record how the two factors $\ch(V_n)$ and $\hat{A}(M)$ in \(K_n\) depend on the scalar and Weyl parts of the
quaternionic curvature. Using the formal roots of $H$ introduced in Subsection~\ref{subsubsec:EH-formalism} and \eqref{eq:ch-SymH}, we can write $\ch(V_n)$ as
\[
 \operatorname{ch}(V_n)
 =(2\sinh h)^{2r_0+2}(2\cosh h)^\epsilon,
 \qquad
 r_0=\left\lfloor\frac n2\right\rfloor,\quad
 \epsilon=n\bmod2.
\]
This is even in \(h\), and hence is a power series only in \(u=h^2\).  Thus the
Chern-character factor contains only the scalar part of the curvature.

Similarly, using the formal roots of $E$ we obtain
\[
 \Ahat(TM)
 =\prod_{a=1}^n
 \frac{(h+x_a)/2}{\sinh((h+x_a)/2)}
 \frac{(h-x_a)/2}{\sinh((h-x_a)/2)}.
\]
Set $p_j:=(-1)^j(q_j+[u]^j)$. By \eqref{eq:q-j}, \(q_j=\sum_a x_a^{2j}\), and hence the above factor is a universal power series in \(u,p_1,p_2,\ldots\). By \eqref{eq:cohomologous-z-q-u}, we can represent $p_j$ by the form $\tfrac{z_j}{2}$.
Give \(u\) weight one and \(z_j\) weight \(j\). Substituting these forms into
\(\Ahat(TM)\operatorname{ch}(V_n)\) and taking the component of degree \(4n\), we obtain a form
cohomologous to $K_n$ of the shape
\begin{equation}\label{eq:Kn-weight-pieces}
       a_{n,0}u^n+
       \sum_{k=1}^n P_{n,k}(z_1,\ldots,z_k)u^{n-k},
\end{equation}
where \(P_{n,k}\) has weight \(k\) and is explicitly of the form
\begin{equation}\label{eq:Kn-weight-expansion}
 P_{n,k}(z_1,\ldots,z_k)
 =\sum_{\substack{e_1,\ldots,e_k\ge0\\
                   e_1+2e_2+\cdots+ke_k=k}}
   c_{n,k;e_1,\ldots,e_k}\,
   z_1^{e_1}\cdots z_k^{e_k}.
\end{equation}
As noted in Remark~\ref{rem:Kn-cohomologous}, we also denote this form by $K_n$. Every summand in \eqref{eq:Kn-weight-pieces} has form degree \(4n\).
The coefficients are fixed rational numbers, computed exactly in
Appendix~\ref{app:reconstruction}.  Consequently
\[
       \int_MK_n=a_{n,0}\int_Mu^n+
       \sum_{k=1}^n\int_M P_{n,k}(z_1,\ldots,z_k)u^{n-k}.
\]
In the dimensions considered below, the scalar term $a_{n,0}\int_Mu^n $ is strictly positive. Indeed, $\int_Mu^n$ is strictly positive by \eqref{eq:kraines-form} and \eqref{eq:u-p1Q}. Note that this is a well-known fact and $4^n\int_M u^n$ is known as Salamon's quaternionic volume; see \cite[\S2]{SemmelmannWeingart2004}. To prove that $a_{n,0}$ is also strictly positive, we evaluate $\int_M K_n$ when $M\cong \HH P^n$ with the standard quaternionic-Kähler metric. By the proof of \cite[Theorem 3.1]{Salamon1982} the quaternionic Weyl curvature of $\HH P^n$ is trivial and therefore we can take $z_j=0$ for all $j$. Moreover, we have $\int_{\HH P^n}u^n=1$ (see \cite[Corollary 3.5]{SemmelmannWeingart2004}) and consequently $$a_{n,0}=\int_{\HH P^n} K_n=\dim \Isom_0 (\HH P^n)-\delta_n.$$
Since $\dim \Isom_0 (\HH P^n)=(n+1)(2n+3)$ (see \cite[Corollary 3.3]{SemmelmannWeingart2004}), we conclude that $$a_{n,0}=\begin{cases}
    2(n+1)^2, \text{ $n$ odd},\\ 2n(n+2), \text{ $n$ even},
\end{cases} $$
which is always strictly positive.

It is therefore enough to prove
\[
       \int_M P_{n,k}(z_1,\ldots,z_k)u^{n-k}\ge0
       \qquad(1\le k\le n).
\]
Thus we need positivity of these particular integrals, not positivity of the formal polynomials \(P_{n,k}\).

\subsection{Local positivity and invariant projection moments}

For \(A\succeq0\), the contraction \(\tr(AY)\) is a scalar four-form.  Its exterior powers
have the positivity needed below; here \((\tr(AY))^k\) denotes the exterior power, not
\(\tr(AY^k)\).

\begin{proposition}\label{prop:PSD-master}
Let \(A\succeq0\) be Hermitian and let \(1\le k\le n\).  Then
\begin{equation}\label{eq:PSD-master}
       \bigl(\tr(AY)\bigr)^k\wedge\Omega^{n-k}\ge0.
\end{equation}
\end{proposition}

\begin{proof}
At a point \(x\in M\), the pair symmetry of the quaternionic Weyl curvature gives
\[
X_x\in \mathfrak{sp}(n)\otimes_{\mathbb R}\operatorname{Herm}(E_x\oplus H_x).
\]
Choosing a real basis \(\{\eta_a\}\) of
\(
\mathfrak{sp}(n)\subset \Lambda^2T_x^*M,
\)
write
\(X_x=\sum_a \eta_a\otimes B_a\), where \( B_a=B_a^*.\)
We have
\(
Y=-X^2
=-\sum_{a,b}\eta_a\wedge\eta_b\otimes B_aB_b,
\)
and hence
\[
\tr(AY)
=
-\sum_{a,b}\eta_a\wedge\eta_b\,\tr(AB_aB_b).
\]
Because the $\eta_a$'s are two-forms, they commute under wedge product.
Thus only the symmetric part of the coefficients contributes, and
\[
\tr(AY)
=
-\sum_{a,b}G_{ab}\,\eta_a\wedge\eta_b,
\]
where
\(
G_{ab}
=
\frac12\tr A(B_aB_b+B_bB_a).
\)
The matrix $G$ is real symmetric and positive semidefinite. Indeed, for
real numbers $t_a$, put
\[
B=\sum_a t_aB_a.
\]
Since each $B_a$ is Hermitian, so is $B$, and therefore $B^2\succeq0$. Since also $A\succeq0$ we have
\[
\sum_{a,b}t_aG_{ab}t_b
=
\tr(AB^2)
\geq 0.
\]
Thus $G\succeq0$, so we may write $G=C^TC$ for some real matrix $C$.
Setting now
\[
\theta_s=\sum_a C_{sa}\eta_a\in\mathfrak{sp}(n),
\]
we obtain summing over $s$
\[
\sum_s\theta_s^2
=
\sum_{a,b}
\left(\sum_s C_{sa}C_{sb}\right)
\eta_a\wedge\eta_b
=
\sum_{a,b}G_{ab}\,\eta_a\wedge\eta_b,
\]
and hence
\[
\tr(AY)=-\sum_s\theta_s^2.
\]
Let $\theta_1,\ldots,\theta_N$ be the forms obtained above.
Take a point $x=(x_1,\ldots,x_N)$ on the unit sphere $S^{N-1}$ and set
\[
\theta_x:=\sum_{s=1}^N x_s\theta_s\in\mathfrak{sp}(n).
\]
Consider first commuting variables $y_1,\ldots,y_N$ and the homogeneous
polynomial
\[
Q(y):=
\int_{S^{N-1}}
\left(\sum_{s=1}^N x_s y_s\right)^{2k}
\,d\sigma(x).
\]
with $d\sigma$ the normalized measure on $S^{N-1}$.
Since the sphere and its measure are invariant under $O(N)$,
\[
Q(Ry)=Q(y)
\qquad\text{for every }R\in O(N).
\]
Thus $Q(y)$ depends only on $|y|^2=y_1^2+\cdots+y_N^2$. Since it is
homogeneous of degree $2k$, there is a constant $c_{N,k}$ such that
\[
Q(y)=c_{N,k}(y_1^2+\cdots+y_N^2)^k.
\]
Taking $y=(1,0,\ldots,0)$ shows that
\[
c_{N,k}
=
\int_{S^{N-1}}x_1^{2k}\,d\sigma(x)>0.
\]

The forms $\theta_s$ have even degree and therefore commute under wedge
product. Hence the same polynomial identity holds with $y_s$ replaced by
$\theta_s$:
\[
\int_{S^{N-1}}\theta_x^{2k}\,d\sigma(x)
=
c_{N,k}\left(\sum_{s=1}^N\theta_s^2\right)^k.
\]
Using $\tr(AY)=-\sum_s\theta_s^2$, we obtain
\[
\bigl(\tr(AY)\bigr)^k
=
\frac{(-1)^k}{c_{N,k}}
\int_{S^{N-1}}\theta_x^{2k}\,d\sigma(x).
\]

Wedging with $\Omega^{n-k}$ gives
\[
\bigl(\tr(AY)\bigr)^k\wedge\Omega^{n-k}
=
\frac{1}{c_{N,k}}
\int_{S^{N-1}}
(-1)^k\theta_x^{2k}\wedge\Omega^{n-k}\,d\sigma(x).
\]
It remains to prove that
\[
(-1)^k\theta^{2k}\wedge\Omega^{n-k}\geq0
\]
for every $\theta\in\mathfrak{sp}(n)$.

Identify $T_xM$ with $\mathbb H^n$. By the quaternionic spectral
theorem, after an $\Sp(n)$-change of quaternionic orthonormal basis we may
write
\[
\theta=\sum_{a=1}^n\lambda_a\beta_a,
\]
where $\beta_a$ is supported on the $a$-th quaternionic line
$V_a\simeq\mathbb H$. With the quaternionic orientation,
\[
\beta_a^2=-2\,\vol_{V_a},
\qquad
\beta_a^3=0.
\]

Write
\[
(-1)^k\theta^{2k}\wedge\Omega^{n-k}
=
F(\lambda_1,\ldots,\lambda_n)\,\vol_g.
\]
Permuting the quaternionic coordinates, or changing the sign of any
$\lambda_a$, is induced by an $\Sp(n)$-change of basis and leaves
$\Omega$ and $\vol_g$ invariant. Thus $F$ is symmetric in the
$\lambda_a$ and even in each variable. Moreover, since $\beta_a^3=0$,
no $\lambda_a$ can occur in $F$ with exponent greater than two.
As $F$ is homogeneous of degree $2k$, it follows that
\[
F(\lambda_1,\ldots,\lambda_n)
=
c_{n,k}
\sum_{\substack{S\subset\{1,\ldots,n\}\\ |S|=k}}
\prod_{a\in S}\lambda_a^2
\]
for some constant $c_{n,k}$.

It remains only to see that $c_{n,k}>0$. Consider the coefficient of
$\lambda_1^2\cdots\lambda_k^2$. In $\theta^{2k}$ it is
\[
\frac{(2k)!}{(2!)^k}\,
\beta_1^2\cdots\beta_k^2
=
(-1)^k(2k)!\,
\vol_{V_1\oplus\cdots\oplus V_k}.
\]
After multiplication by $(-1)^k$, this is positive. Since this already
has top degree on $V_1\oplus\cdots\oplus V_k$, only the restriction of
$\Omega^{n-k}$ to the quaternionic orthogonal complement contributes.
By our normalization of the Kraines form,
\[
\left.\Omega^{n-k}\right|_{V_{k+1}\oplus\cdots\oplus V_n}
=
(2n-2k+1)!\,
\vol_{V_{k+1}\oplus\cdots\oplus V_n},
\]
which is positive. Hence $c_{n,k}>0$, and therefore
\[
(-1)^k\theta^{2k}\wedge\Omega^{n-k}\geq0.
\]

Applying this to every $\theta_x$ in the sphere average gives
\[
\bigl(\tr(AY)\bigr)^k\wedge\Omega^{n-k}\geq0,
\]
as required.
\end{proof}

The proposition is pointwise.  A fixed numerical matrix \(A\) does not in general give a
global form: for example, projection onto the first \(\ell\) coordinates depends on the local
frame.  To remove this choice, let \(P_\ell^0\) be that projection in \(\CC^r\), where
\(r=2n+2\), and average over its conjugates.  With normalized Haar measure, define
\begin{equation}\label{eq:projection-expectation}
 \begin{split}
 M_{k,\ell}
   &:=\int_{U(r)}
       \bigl(\tr(gP_\ell^0g^{-1}Y)\bigr)^k\,dg,\\
 \mu_{k,\ell}
   &:=\int_M M_{k,\ell}u^{n-k},
 \qquad
 \mu_0:=\int_Mu^n.
 \end{split}
\end{equation}
Thus \(M_{k,\ell}\) averages over all \(\ell\)-dimensional subspaces.  If the local frame is
changed by \(h\in U(r)\), Haar invariance gives
\[
       M_{k,\ell}(h^{-1}Yh)=M_{k,\ell}(Y).
\]
Hence the local expressions patch to a global \(4k\)-form.  Equivalently,
\(M_{k,\ell}\) is a conjugation-invariant polynomial in \(Y\), and therefore a polynomial
in the traces \(z_j=\tr(Y^j)\).  When \(\ell=r\), the projection is the identity and no
averaging is needed: \(M_{k,r}=z_1^k\).

For each \(g\in U(r)\), the matrix
\[
        A_g:=gP_\ell^0g^{-1}
\]
is an orthogonal projection and hence is positive semidefinite.  Proposition~\ref{prop:PSD-master}
therefore gives
\[
        \bigl(\tr(A_gY)\bigr)^k\wedge\Omega^{n-k}\ge0.
\]

\begin{corollary}[Positivity of the projection moments]\label{cor:mu-positive}
For \(1\le k\le n\) and \(1\le\ell\le r\),
\[
        \mu_0>0,\qquad \mu_{k,\ell}\ge0.
\]
\end{corollary}

\begin{proof}
By \eqref{eq:u-p1Q}, \(u=c\Omega\) for a constant \(c>0\).  Averaging the preceding
inequality gives the pointwise inequality
\begin{equation}\label{eq:projection-moment-positive}
 M_{k,\ell}u^{n-k}
 =c^{n-k}\int_{U(r)}
       \bigl(\tr(gP_\ell^0g^{-1}Y)\bigr)^k
       \wedge\Omega^{n-k}\,dg
 \ge0.
\end{equation}
Consequently
\[
 \mu_{k,\ell}
 =c^{n-k}\int_M\int_{U(r)}
       \bigl(\tr(gP_\ell^0g^{-1}Y)\bigr)^k
       \wedge\Omega^{n-k}\,dg\ge0.
\]
This is an iterated integral: the inner integral is the invariant average at each point of
\(M\).  Finally, \(u^n=c^n\Omega^n\) is strictly positive, so \(\mu_0>0\).
\end{proof}

No sign is asserted for the polynomial \(M_{k,\ell}\), or for its coefficients in the
\(z_j\).  Positivity concerns the top-degree form \(M_{k,\ell}u^{n-k}\), and hence its
integral \(\mu_{k,\ell}\).  Since \(M_{k,\ell}\) is an invariant Chern--Weil polynomial, it
is closed and
\begin{equation}\label{eq:mu-topological}
       \mu_{k,\ell}
       =\left\langle[M_{k,\ell}]\,u^{n-k},[M]\right\rangle
\end{equation}
is a characteristic number.  The explicit invariant-polynomial formula is recorded in
Appendix~\ref{app:reconstruction}.

\subsection{The algebraic bridge to \(K_n\)}
\label{subsec:algebraic-bridge}

Let \(\mathcal V_k\) be the real vector space of weight-\(k\) polynomials in the \(z_j\):
\[
 \mathcal V_k
 =\operatorname{span}_{\mathbb R}
 \left\{
 z_1^{e_1}\cdots z_k^{e_k}:
 e_1,\ldots,e_k\ge0,\quad e_1+2e_2+\cdots+ke_k=k
 \right\}.
\]
Both \(P_{n,k}\) and the \(M_{k,\ell}\) belong to \(\mathcal V_k\).  The projection moments
generate the cone
\[
 \mathcal C_k^{\mathrm{proj}}
 :=\left\{
       \sum_{\ell=1}^r b_\ell M_{k,\ell}:
       b_\ell\ge0
   \right\}
 \subset\mathcal V_k.
\]
For the given manifold, define the linear functional
\[
       L_k(Q):=\int_MQ\,u^{n-k},\qquad Q\in\mathcal V_k.
\]
Corollary~\ref{cor:mu-positive} says that
\[
       L_k(M_{k,\ell})=\mu_{k,\ell}\ge0,
\]
and hence, by linearity,
\[
       L_k(\mathcal C_k^{\mathrm{proj}})\subset\mathbb R_{\ge0}.
\]
Thus \(\mathcal C_k^{\mathrm{proj}}\) consists of polynomials whose associated characteristic
numbers are known to be nonnegative; its elements need not be coefficientwise positive.

The required assertion at weight \(k\) is
\begin{equation}\label{eq:positive-change-of-coordinates}
       P_{n,k}\in\mathcal C_k^{\mathrm{proj}},
       \qquad\text{that is,}\qquad
       P_{n,k}=\sum_{\ell=1}^r b_{n,k,\ell}M_{k,\ell},
       \quad b_{n,k,\ell}\ge0.
\end{equation}
If this identity holds, then
\[
 \int_M P_{n,k}u^{n-k}
 =\sum_{\ell=1}^r b_{n,k,\ell}
       \int_M M_{k,\ell}u^{n-k}
 =\sum_{\ell=1}^r b_{n,k,\ell}\mu_{k,\ell}\ge0.
\]

Membership in this cone is decided by an ordinary coefficient comparison.  Write
\[
 M_{k,\ell}
 =\sum_{\substack{e_1,\ldots,e_k\ge0\\
                   e_1+2e_2+\cdots+ke_k=k}}
   d_{k,\ell;e_1,\ldots,e_k}z_1^{e_1}\cdots z_k^{e_k}.
\]
Comparing this with \eqref{eq:Kn-weight-expansion} gives the finite linear system
\[
 c_{n,k;e_1,\ldots,e_k}
 =\sum_{\ell=1}^r b_{n,k,\ell}d_{k,\ell;e_1,\ldots,e_k}.
\]
The common weight makes this comparison possible, but does not guarantee a nonnegative solution. The coefficients \(d_{k,\ell;e_1,\ldots,e_k}\) themselves may have either sign.

At quadratic Weyl weight the comparison is already completely visible:
\[
       M_{2,r}=z_1^2,
       \qquad
       M_{2,1}=\frac{z_1^2+z_2}{r(r+1)}.
\]
Hence
\[
       Az_1^2+Bz_2
       =(A-B)M_{2,r}+B\,r(r+1)M_{2,1}.
\]
Thus this polynomial belongs to \(\mathcal C_2^{\mathrm{proj}}\) whenever
\(B\ge0\) and \(A-B\ge0\).
At higher weights, more monomials and moments appear, but the problem remains to express the resulting polynomial as a nonnegative linear combination. In Appendix~\ref{app:reconstruction} we give the exact rational identities.

For \(5\le n\le8\), the projection moments give the required nonnegative decompositions
directly.  In dimensions \(9\) and \(10\) we use additional averages arising from more general positive semidefinite matrices (PSD) \(A\), still governed by Proposition~\ref{prop:PSD-master}.  In dimensions \(11\) and \(12\), the whole PSD matrix-moment cone fails at Weyl weight four, and a genuinely larger family of positive scalar contractions of the curvature is required.  The following sections introduce these enlargements only when they become necessary.

\subsection{$1<n\leq 5$}
\label{subsec:low-dimensions}
The cases \(n\le5\) are included to display the curvature calculation in dimensions where it
can be followed entirely by hand.  They are not asserted here as independent new proofs of the
LeBrun--Salamon conjecture: the present symmetry reduction uses the known low-dimensional
cases.  The case \(n=5\) is the first new application.

For completeness, the virtual object already defined by its character can be written explicitly
as a finite sum of symmetric powers.  If \(n\) is even, then
\begin{equation}\label{eq:V-even}
 V_n=
 \sum_{j=0}^{(n+2)/2}
 (-1)^j\binom{n+3}{j}\Sym^{n+2-2j}H.
\end{equation}
If \(n\) is odd, then
\begin{equation}\label{eq:V-odd}
 V_n=
 \sum_{j=0}^{(n+1)/2}
 (-1)^j
 \left(\binom{n+1}{j}-\binom{n+1}{j-2}\right)
 \Sym^{n+2-2j}H,
\end{equation}
where a binomial coefficient with negative lower index is zero.  In particular,
\begin{align*}
 V_2&=\Sym^4H-5\Sym^2H+10,\\
 V_3&=\Sym^5H-4\Sym^3H+5H,\\
 V_4&=\Sym^6H-7\Sym^4H+21\Sym^2H-35,\\
 V_5&=\Sym^7H-6\Sym^5H+14\Sym^3H-14H.
\end{align*}
Since the roots of \(\Sym^mH\) are \(mh,(m-2)h,\ldots,-mh\), adding the Chern characters of
these virtual sums recovers the formula for \(\operatorname{ch}(V_n)\) in
Subsection~\ref{subsec:curvature-polynomial}.  The four expansions needed below are
\begin{align}
 \operatorname{ch}(V_2)&=16u^2+O(u^3),\label{eq:ch2}\\
 \operatorname{ch}(V_3)&=32u^2+\frac{112}{3}u^3+O(u^4),\label{eq:ch3}\\
 \operatorname{ch}(V_4)&=64u^3+64u^4+O(u^5),\label{eq:ch4}\\
 \operatorname{ch}(V_5)&=128u^3+192u^4+\frac{1936}{15}u^5+O(u^6).
 \label{eq:ch5}
\end{align}

In the root formula for \(\Ahat(TM)\) of Subsection~\ref{subsec:curvature-polynomial}, each
factor is \(f(h\pm x_a)\) with
\[
 f(s)=\frac{s/2}{\sinh(s/2)}=1-\frac{s^2}{24}+O(s^4),
\]
so the component of weight at most one is
\[
 \Ahat(TM)
 =1-\frac1{24}\sum_{a=1}^n
 \bigl((h+x_a)^2+(h-x_a)^2\bigr)+O(\text{weight }2).
\]
 Thus this is
\[
 1-\frac1{12}\left(nu+\sum_{a=1}^n x_a^2\right)
 +O(\text{weight }2).
\]
Since \(z_1/2=p_1=-(q_1+u)\), this gives
\begin{equation}\label{eq:Ahat-one}
\Ahat(TM)
 =1+\frac{z_1}{24}-\frac{n-1}{12}u
 +O(\text{weight }2).
\end{equation}

Only \(n=5\) requires the weight-two term.  We compute it directly from the
product.  Since
\[
 f(s)=1-\frac{s^2}{24}+\frac{7s^4}{5760}+O(s^6),
\]
direct multiplication gives
\[
 \begin{aligned}
 \Ahat(TM)={}&1-\frac1{24}\sum_{a=1}^5
 \bigl((h+x_a)^2+(h-x_a)^2\bigr)\\
 &+\frac1{1152}\left[
 \sum_{a=1}^5\bigl((h+x_a)^2+(h-x_a)^2\bigr)
 \right]^2\\
 &+\frac1{2880}\sum_{a=1}^5
 \bigl((h+x_a)^4+(h-x_a)^4\bigr)
 +O(\text{weight }3).
 \end{aligned}
\]
In quaternionic dimension five, using the Chern--Weil identities recalled above,
\[
 \sum_{a=1}^5\bigl((h+x_a)^2+(h-x_a)^2\bigr)=8u-z_1,
 \]
and
\[
 \sum_{a=1}^5\bigl((h+x_a)^4+(h-x_a)^4\bigr)
 =-4u^2-6z_1u+z_2.
\]
Substitution gives
\begin{equation}\label{eq:Ahat-two}
 \begin{aligned}
 \Ahat(TM)={}&1+\left(\frac{z_1}{24}-\frac{u}{3}\right)\\
 &+\left(
 \frac{z_1^2}{1152}+\frac{z_2}{2880}
 -\frac{23z_1u}{1440}+\frac{13u^2}{240}
 \right)+O(\text{weight }3)
 \end{aligned}
\end{equation}

In each case below we compute \(K_n\) explicitly and rewrite it, by the coefficient comparison
of Subsection~\ref{subsec:algebraic-bridge}, as a nonnegative linear combination of
\(u^n\) and the forms \(M_{k,\ell}u^{n-k}\), where \(r=2n+2\).  Integrating over \(M\) expresses
\(d-\delta_n\) through \(\mu_0>0\) and the projection moments \(\mu_{k,\ell}\ge0\) of
Corollary~\ref{cor:mu-positive}.  Only the moments \(M_{k,r}=z_1^k\) and \(M_{2,1}\) occur.

\subsubsection*{Quaternionic dimension two}

Here the top degree has weight two, while \eqref{eq:ch2} already begins in weight two.
Consequently only the constant term of \(\Ahat(TM)\) contributes:
\[
 \boxed{K_2=16u^2.}
\]
Thus, with \(\mu_0=\int_Mu^2>0\),
\[
 \int_MK_2=16\int_Mu^2=16\mu_0,
 \qquad\text{hence}\qquad d-5=16\mu_0>0.
\]
No matrix \(A\) is used in this case.

\subsubsection*{Quaternionic dimension three}

Now \eqref{eq:ch3} begins in weight two, so only weights zero and one of
\(\Ahat(TM)\) can contribute to weight three.  Substituting \(n=3\) in
\eqref{eq:Ahat-one},
\begin{align*}
K_3
 &=\frac{112}{3}u^3
   +32u^2\left(\frac{z_1}{24}-\frac{u}{6}\right)\\
 &=\boxed{32u^3+\frac43z_1u^2}.
\end{align*}
Here \(r=8\).  Since \(M_{1,8}=z_1\), this is the identity
\[
 K_3=32u^3+\frac43M_{1,8}u^2.
\]
Integrating it term by term gives
\[
 \int_MK_3=32\mu_0+\frac43\mu_{1,8},
 \qquad\text{hence}\qquad
 d-4=32\mu_0+\frac43\mu_{1,8}>0.
\]

\subsubsection*{Quaternionic dimension four}

Equation \eqref{eq:ch4} begins in weight three, so again only the weight-one term of
\(\Ahat(TM)\) is required.  From \eqref{eq:Ahat-one},
\begin{align*}
 K_4
 &=64u^4
   +64u^3\left(\frac{z_1}{24}-\frac{u}{4}\right)\\
 &=\boxed{48u^4+\frac83z_1u^3}.
\end{align*}
Here \(r=10\).  Since \(M_{1,10}=z_1\), we have the identity of top-degree forms
\[
 K_4=48u^4+\frac83M_{1,10}u^3.
\]
Integrating gives
\[
 \int_MK_4=48\mu_0+\frac83\mu_{1,10},
 \qquad\text{hence}\qquad
 d-7=48\mu_0+\frac83\mu_{1,10}>0.
\]

\subsubsection*{Quaternionic dimension five}

Here \eqref{eq:ch5} begins in weight three.  To obtain weight five we therefore need
weights zero, one and two of \(\Ahat(TM)\).  Substituting
\eqref{eq:Ahat-two} gives
\begin{align*}
K_5={}&\frac{1936}{15}u^5
 +192u^4\left(\frac{z_1}{24}-\frac{u}{3}\right)\\
&+128u^3\left(
 \frac{z_1^2}{1152}+\frac{z_2}{2880}
 -\frac{23z_1u}{1440}+\frac{13u^2}{240}
 \right)\\
={}&\boxed{
72u^5+\frac{268}{45}z_1u^4
+\left(\frac19z_1^2+\frac{2}{45}z_2\right)u^3}.
\end{align*}

Here \(r=12\).  The linear term is \(\frac{268}{45}M_{1,12}\).  The new term is the
weight-two polynomial in parentheses, and for it we use the two moments
\[
 M_{2,12}=z_1^2,
 \qquad
 M_{2,1}=\frac{z_1^2+z_2}{12\cdot13}.
\]
To find the required linear combination, write
\[
 \frac19z_1^2+\frac{2}{45}z_2
 =aM_{2,12}+bM_{2,1}.
\]
Substituting the two formulas above, the right-hand side becomes
\[
 \left(a+\frac{b}{156}\right)z_1^2+\frac{b}{156}z_2.
\]
Equating the coefficients of \(z_1^2\) and \(z_2\) gives
\[
 a+\frac{b}{156}=\frac19,
 \qquad
 \frac{b}{156}=\frac{2}{45}.
\]
Therefore \(a=1/15\) and \(b=104/15\), both positive, and hence
\[
 \frac19z_1^2+\frac{2}{45}z_2
 =\frac1{15}M_{2,12}+\frac{104}{15}M_{2,1}.
\]
Together with \(M_{1,12}=z_1\), this rewrites the previously computed
\(K_5\) itself as
\[
 \boxed{
 K_5=72u^5+\frac{268}{45}M_{1,12}u^4
 +\left(\frac1{15}M_{2,12}
 +\frac{104}{15}M_{2,1}\right)u^3.}
\]
This is the required linear-combination identity.  Integrating it term by term
and using the definitions of the \(\mu_{k,\ell}\) gives
\[
 \int_MK_5
 =72\mu_0+\frac{268}{45}\mu_{1,12}
 +\frac1{15}\mu_{2,12}+\frac{104}{15}\mu_{2,1}.
\]
Since all four quantities on the right are nonnegative and \(\mu_0>0\),
\[
 \boxed{
 d-6=72\mu_0+\frac{268}{45}\mu_{1,12}
 +\frac1{15}\mu_{2,12}+\frac{104}{15}\mu_{2,1}>0.}
\]
After multiplication by \(45\), this reads
\begin{equation}\label{eq:n5mu}
45(d-6)=3240\mu_0+268\mu_{1,12}+3\mu_{2,12}+312\mu_{2,1}.
\end{equation}

These calculations show the first change in the method.  Dimension two is purely scalar;
dimensions three and four use only \(A=I_r\) with \(k=1\); dimension five also uses
\(A=I_{12}\) with \(k=2\), and is the first case in which the nontrivial rank-one family
\(A_g=gP_1^0g^{-1}\) is needed to control a combination of \(z_1^2\) and \(z_2\).
\section{The first cases: projection moments suffice}
\label{sec:first-cases}

In Appendix~\ref{app:reconstruction-5-8} we reconstruct the already-defined characteristic number $K_n$ in the corrected Weyl power sums.  The Schur--Weyl identities enable us to convert those power sums into a combination of the projection-moment generators.  The following exact identities are the resulting decompositions. The positivity of the coefficients, rather than the reconstruction itself, is the point used in the proof.

The first four new dimensions require no positivity beyond the projection moments
$\mu_{k,\ell}$.  The case $n=5$ was computed explicitly in
Subsection~\ref{subsec:low-dimensions}, giving \eqref{eq:n5mu}; the cases $n=6,7,8$ follow the
same pattern.  Note that cubic moments first appear at $n=7$.

For $n=6$ ($r=14$), exact coefficient comparison gives
\begin{equation}\label{eq:n6mu}
45(d-9)=4320\mu_0+416\mu_{1,14}+6\mu_{2,14}+840\mu_{2,1}.
\end{equation}
For $n=7$ ($r=16$), after clearing a factor two,
\begin{align}\label{eq:n7mu}
1890(d-8)={}&241920\mu_0+28080\mu_{1,16}+668\mu_{2,16}
 +100096\mu_{2,1}\notag\\
&+3\mu_{3,16}+6528\mu_{3,1}+4080\mu_{3,2}.
\end{align}
For $n=8$ ($r=18$),
\begin{align}\label{eq:n8mu}
4725(d-11)={}&756000\mu_0+96720\mu_{1,18}+2710\mu_{2,18}
 +485640\mu_{2,1}\notag\\
&+15\mu_{3,18}+43320\mu_{3,1}+29070\mu_{3,2}.
\end{align}
By Corollary~\ref{cor:mu-positive}, the right-hand sides of these three identities are strictly positive and therefore $d>9$, $d>8$ and $d>11$ respectively.

Combining these inequalities and the inequality $d>6$ from \eqref{eq:n5mu} with
Theorem~\ref{M-min-symm} excludes a minimal counterexample in each of the dimensions
$5\le n\le8$.  At $n=9$ a Weyl-weight-four component appears for the first time, and the projection-moment decomposition used above no longer produces an identity where all the coefficients are positive. 

\section{Dimensions nine and ten: a positive integral for the highest Weyl term}
\label{sec:nine-ten}

For $n\leq 8$, every Weyl-weight component of $K_n$ can be written as a nonnegative linear combination of the projection moments $M_{k,\ell}$. In dimensions $9$ and $10$, the components of Weyl weights $1,2,3$ are handled in exactly the same way. The only new issue is the highest Weyl-weight term, of weight four.

We first explain why the particular quartic polynomial which occurs is natural. Recall that
\[
\widehat A(TM)
=
\prod_{a=1}^n
f(h+x_a)f(h-x_a),
\qquad
f(s)=\frac{s/2}{\sinh(s/2)}.
\]
Write
\[
\log f(s)=\sum_{j\geq1}\ell_j s^{2j},
\qquad
a_j:=(-1)^j\ell_j
=\frac{\zeta(2j)}{j(4\pi^2)^j}>0.
\]
After replacing the formal $E$-root power sums by
\[
p_j=\frac{z_j}{2}=(-1)^j(q_j+u^j),
\]
the part of $\log\widehat A(TM)$ of Weyl weight $j$ containing no factor of $u$ is $a_jz_j$. Thus the purely Weyl part of the $\widehat A$-series is encoded by
\begin{equation}\label{eq:Fseries}
F(Y;t)
:=
\exp\!\left(\sum_{j\geq1}a_jz_jt^j\right)
=
\sum_{k\geq0}F_k(Y)t^k,
\end{equation}
where $t$ records Weyl weight.

For $n=9$, $\operatorname{ch}(V_9)$ begins in $u^5$, and for $n=10$ it begins in $u^6$. Hence the largest possible Weyl weight in either case is four. Appendix~\ref{app:reconstruction-9-10} computes the characteristic numbers explicitly and shows that
\[
[K_9]_4=2^{11}F_4,
\qquad
[K_{10}]_4=2^{12}F_4.
\]
Thus it remains only to prove
\[
\int_M F_4(Y)u^{n-4}\geq0.
\]

This follows from Proposition~\ref{prop:PSD-master} by an elementary positive integral representation. Euler's product for $\sinh$ gives
\[
F(Y;t)
=
\prod_{m\geq1}\det(I-q_mtY)^{-1},
\qquad
q_m=\frac1{4\pi^2m^2}.
\]
To understand one factor, start from
\begin{equation}\label{eq:positive-integral-one-variable}
\frac1{1-s}
=
\int_0^\infty e^{-u}e^{su}\,du
=
\frac1\pi\int_{\mathbf C}e^{-|v|^2}e^{s|v|^2}\,dv.
\end{equation}
Let for the moment $Y_0$ be an ordinary Hermitian $r\times r$ matrix with eigenvalues
$y_1,\ldots,y_r$. Applying \eqref{eq:positive-integral-one-variable} with
$s=qty_i$ and multiplying over $i$ gives
\begin{equation}\label{eq:inverse-det-positive-integral}
\det(I-qtY_0)^{-1}
=
\frac1{\pi^r}
\int_{\mathbf C^r}
e^{-|v|^2}
e^{\,t\operatorname{tr}(qvv^*Y_0)}\,dv.
\end{equation}
Indeed, in a basis diagonalizing $Y_0$,
\[
v^*Y_0v=\sum_i y_i|v_i|^2
=\operatorname{tr}(vv^*Y_0).
\]
Both sides of \eqref{eq:inverse-det-positive-integral} are conjugation-invariant power series whose coefficients are universal polynomials in the traces $\operatorname{tr}(Y_0^j)$, so the same coefficient identity applies to the Chern--Weil variable $Y$.

For the finite product over $m\leq N$, introduce vectors
$v_1,\ldots,v_N\in\mathbf C^r$ and set
\[
A_N:=\sum_{m=1}^Nq_mv_mv_m^*.
\]
Since each $v_mv_m^*$ is positive semidefinite,
\[
A_N\succeq0.
\]
Multiplying \eqref{eq:inverse-det-positive-integral} for $m=1,\ldots,N$ gives
\[
\prod_{m=1}^N\det(I-q_mtY)^{-1}
=
\frac1{\pi^{rN}}
\int
e^{-\sum_m|v_m|^2}
e^{\,t\operatorname{tr}(A_NY)}
\,dv_1\cdots dv_N.
\]
Expanding the final exponential, the coefficient of $t^k$ is
\[
\frac1{k!\pi^{rN}}
\int
e^{-\sum_m|v_m|^2}
\bigl(\operatorname{tr}(A_NY)\bigr)^k
\,dv_1\cdots dv_N.
\]
Letting $N\to\infty$ therefore gives
\begin{equation}\label{eq:Fk-positive-integral}
F_k(Y)
=
\frac1{k!}
\lim_{N\to\infty}
\frac1{\pi^{rN}}
\int
e^{-\sum_m|v_m|^2}
\bigl(\operatorname{tr}(A_NY)\bigr)^k
\,dv_1\cdots dv_N.
\end{equation}

For every choice of the vectors $v_m$, the matrix $A_N$ is positive semidefinite. Proposition~\ref{prop:PSD-master} therefore gives
\[
\bigl(\operatorname{tr}(A_NY)\bigr)^k
\wedge\Omega^{n-k}\geq0.
\]
The factor $e^{-\sum_m|v_m|^2}$ in \eqref{eq:Fk-positive-integral} is positive, so integration preserves the inequality. Hence
\[
F_k(Y)\wedge\Omega^{n-k}\geq0.
\]
Since $u$ is a positive multiple of $\Omega$,
\[
\Phi_k^{(n)}
:=
\int_MF_k(Y)u^{n-k}\geq0.
\]

Thus the new weight-four term in dimensions nine and ten is positive for the same geometric reason as the earlier matrix moments: it is a positive integral of expressions already covered by Proposition~\ref{prop:PSD-master}. The integral representation is only a convenient way of recognizing this particular polynomial.

For $n=9$, with $r=20$, the exact coefficient calculation of Appendix~\ref{app:reconstruction-9-10} gives
\begin{equation}\label{eq:n9}
\begin{aligned}
d-10={}&200\mu_0
+\frac{180512}{315}\mu_{1,1}
+\frac{8752}{45}\mu_{2,1}
+\frac{14704}{14175}\mu_{2,20}\\
&+\frac{2992}{405}\mu_{3,1}
+\frac{2090}{81}\mu_{3,2}
+\frac{239}{28350}\mu_{3,20}
+2^{11}\Phi_4^{(9)}.
\end{aligned}
\end{equation}
Hence $d>10$.

For $n=10$, with $r=22$,
\begin{equation}\label{eq:n10}
\begin{aligned}
d-13={}&240\mu_0
+\frac{182336}{225}\mu_{1,1}
+\frac{1495736}{4725}\mu_{2,1}
+\frac{21278}{14175}\mu_{2,22}\\
&+\frac{4048}{4725}\mu_{3,1}
+\frac{23276}{405}\mu_{3,2}
+\frac{194}{14175}\mu_{3,22}
+2^{12}\Phi_4^{(10)}.
\end{aligned}
\end{equation}
Hence $d>13$.

All terms on the right-hand sides are nonnegative, while the scalar term is strictly positive. Theorem~\ref{M-min-symm} therefore excludes a minimal counterexample in quaternionic dimensions nine and ten.

This PSD-moment mechanism stops in dimensions eleven and twelve: the exact separating
functionals of Appendix~\ref{app:obstruction-psd} show that the new weight-four targets lie outside
the entire PSD matrix-moment cone.  This forces the scalar-orbital enlargement used next.

\section{Dimensions eleven and twelve: scalar orbital moments}
\label{sec:eleven-twelve}

The next step is genuinely different.  In dimensions eleven and twelve the components of
Weyl weights at most three are still handled by projection moments, while the highest
Weyl-weight term is again
\[
        [K_{11}]_5=2^{13}F_5,
        \qquad
        [K_{12}]_5=2^{14}F_5,
\]
and is therefore nonnegative by Section~\ref{sec:nine-ten}.  The only new difficulty is the component of Weyl weight four.

One might hope to handle this term by choosing more projection ranks, or more general
positive-semidefinite matrices $A$ in Proposition~\ref{prop:PSD-master}.  This is impossible: Appendix~\ref{app:obstruction-psd} gives an exact separating functional showing that the weight-four
components of both $K_{11}$ and $K_{12}$ lie outside the entire PSD matrix-moment cone.
We therefore return to the more basic positivity used in the proof of Proposition~\ref{prop:PSD-master},
before $X$ was squared.

\subsection{Signed scalar orbital positivity}

Recall from the proof of Proposition~\ref{prop:PSD-master} that for every real two-form
$\theta\in\mathfrak{sp}(n)$,
\[
        (-1)^k\theta^{2k}\wedge\Omega^{n-k}\geq0.
\]
We apply this directly to scalar contractions of $X$.  Let $B$ be a fixed Hermitian matrix
on the nonzero quaternionic block of $X$, and for $g\in\Sp(n)$ put
\[
        \theta_g:=\tr(BgXg^{-1}).
\]
If $J$ denotes the quaternionic structure and $\tau(B)=JBJ^{-1}$, then
$\tau(gXg^{-1})=-gXg^{-1}$, so
\[
 \tr(BgXg^{-1})
 =
 \tr\left(\frac{B-\tau(B)}2\,gXg^{-1}\right).
\]
Thus we may take $B\in i\mathfrak{sp}(n)$, with eigenvalues
\[
        \pm\sqrt{a_1},\ldots,\pm\sqrt{a_n},
        \qquad a_i\geq0.
\]

Averaging the even powers of $\theta_g$ over $\Sp(n)$ gives, up to a positive normalization,
a pointwise-positive invariant polynomial depending only on
$a=(a_1,\ldots,a_n)$.  We normalize it by the type-$C$ Harish--Chandra formula
\begin{equation}\label{eq:orbital-def}
 O_k^{(n)}(a;p)=
 4^k\sum_{\lambda\vdash k}s_\lambda(a)s_\lambda(p)
 \prod_{i=1}^{\ell(\lambda)}
 \frac{(2(n-i)+1)!}{(2(\lambda_i+n-i)+1)!},
\end{equation}
where $s_\lambda(p)$ denotes the Schur function expressed in the power-sum variables
$p_j=z_j/2$; see \cite{HarishChandra1957}.  Hence
\begin{equation}\label{eq:orbital-def positivity}
       O_k^{(n)}(a;p)\,u^{n-k}\ge0,
       \qquad
       \omega_k^{(n)}(a):=\int_MO_k^{(n)}(a;p)u^{n-k}\ge0.
\end{equation}

Appendix~\ref{app:orbital-inclusion} shows that, after invariant averaging, every PSD moment $(\tr AY)^k$ is a positive mixture of scalar orbital moments.  Thus the projection moments and the terms $F_k$ of Section~\ref{sec:nine-ten} already lie in this larger cone; they could have been treated this way in dimensions nine and ten, but the smaller PSD cone was sufficient there.

\subsection{The quartic replacements}

Let $Q_{11,4}$ and $Q_{12,4}$ denote the Weyl-weight-four components of $K_{11}$ and
$K_{12}$.  Appendix~\ref{app:reconstruction-11-12} gives the exact five-dimensional linear-algebra calculation leading
to the following positive decompositions:
\begin{align}\label{eq:Q11orb}
Q_{11,4}={}&
 \frac{1008851484437}{2142693}O_4^{(11)}(1,1)
 +\frac{151204920425}{714231}O_4^{(11)}(2,1,1)\notag\\
&+\frac{10770035384}{714231}O_4^{(11)}(5,1,1,1,1)
 +\frac{51654964}{6428079}O_4^{(11)}(10,1^8)\notag\\
&+\frac{252882056}{1530495}O_4^{(11)}(10,1^{10}),
\end{align}
and
\begin{align}\label{eq:Q12orb}
Q_{12,4}={}&
 \frac{2298362246}{9933}O_4^{(12)}(1,1)
 +\frac{194400388184}{148995}O_4^{(12)}(1,1,1)\notag\\
&+\frac{7327263802}{9933}O_4^{(12)}(2,1,1)
 +\frac{178334456}{148995}O_4^{(12)}(10,1^8)\notag\\
&+\frac{4485216}{3311}O_4^{(12)}(10,1^9).
\end{align}
Here omitted spectral coordinates are zero and $1^j$ denotes $j$ entries equal to one.
All coefficients are positive, and therefore
\[
        \int_MQ_{11,4}u^7\ge0,
        \qquad
        \int_MQ_{12,4}u^8\ge0.
\]

The remaining components may still be written in the convenient projection-moment
coordinates.  The exact reconstruction gives
\begin{align}\label{eq:n11}
 d-12={}&288\mu_0+\frac{2482652}{51975}\mu_{1,24}
 +\frac{3094720}{6237}\mu_{2,1}+\frac{214814}{93555}\mu_{2,24}\notag\\
&+\frac{164360}{2673}\mu_{3,2}+\frac{10048}{891}\mu_{3,3}
 +\frac{922051}{93555}\mu_{3,4}
 +\int_MQ_{11,4}u^7+2^{13}\Phi_5^{(11)},
\end{align}
and
\begin{align}\label{eq:n12}
 d-15={}&336\mu_0+\frac{3050776}{51975}\mu_{1,26}
 +\frac{12506936}{17325}\mu_{2,1}+\frac{1445966}{467775}\mu_{2,26}\notag\\
&+\frac{472784}{155925}\mu_{3,1}+\frac{152711}{18711}\mu_{3,2}
 +\frac{1771432}{18711}\mu_{3,3}
 +\int_MQ_{12,4}u^8+2^{14}\Phi_5^{(12)}.
\end{align}
Every term on the right is nonnegative and the scalar term is strictly positive.  Hence
\[
        d>12\quad(n=11),
        \qquad
        d>15\quad(n=12).
\]
Together with Theorem~\ref{M-min-symm}, this proves Theorem~\ref{thm:main} up to $n=12$.

\section{Extension through fourteen: Hodge squares}
\label{sec:hodge-extension}

The scalar orbital cone is no longer sufficient in quaternionic dimension thirteen. The weight-four component of $K_{13}$ lies outside it, as shown by the exact separating functional in Appendix~\ref{app:obstruction-orbital}.  We use instead a standard consequence of the Hodge--Riemann relations on the twistor space.

\subsection{Hodge--Riemann positivity}

Let $\pi:Z\to M$ be the twistor fibration and let $L$ be its contact line bundle.  Set
\[
        h=\frac12c_1(L)\in H^2(Z,\mathbb Q).
\]
We use the same notation $h$ as for the formal Chern root of $H$.  The standard twistor-space identities are
\begin{equation}\label{eq:twistor-h}
        h^2=\pi^*u,\qquad \pi_*h=1.
\end{equation}
The first identifies the square of $h$ with the pullback of $u=-c_2(H)$, while the second follows from
\[
        L|_{\pi^{-1}(x)}\simeq\mathcal O_{\mathbb{CP}^1}(2),
\]
so that $\int_{\pi^{-1}(x)}h=1$.  Up to a positive constant, $h$ is the K\"ahler class of $Z$.

The Hodge--Riemann bilinear relations on $Z$ imply the definiteness of the quaternionic intersection form: for even $r$,
\begin{equation}\label{eq:quaternionic-intersection}
        Q_r(x,y)
        =
        (-1)^{r/2}
        \int_Mxy\,u^{\,n-r/2},
        \qquad
        x,y\in H^r(M,\mathbb R),
\end{equation}
is positive definite.  This is the classical quaternionic intersection-form result of Fujiki \cite[Theorem 3.27]{Fujiki1987}; see also \cite[Theorem~1.2]{Amann2012} for a convenient formulation in the positive quaternionic Kähler setting. Indeed, the corresponding intersection form is obtained on $Z$ by pulling classes back from $M$ and using \eqref{eq:twistor-h}.  In particular, for $a\in H^{4k}(M,\mathbb R)$,
\[
\begin{aligned}
        \int_Z(\pi^*a)^2h^{2n+1-4k}
        &=
        \int_Z\pi^*\!\left(a^2u^{n-2k}\right)h  \\
        &=
        \int_Ma^2u^{n-2k}.
\end{aligned}
\]
Since $r=4k$ has positive Hodge--Riemann sign, we obtain
\begin{equation}\label{eq:hodge-square}
        \int_Ma^2u^{\,n-2k}\ge0,
\end{equation}
with strict inequality unless $a=0$ in cohomology.

\subsection{Quaternionic dimension thirteen}\label{subsec:hodge-thirteen}

Write the index density, as above, in Weyl weights:
\[
        K_{13}
        =
        392u^{13}
        +\sum_{k=1}^{6}P_{13,k}u^{13-k},
\]
where $P_{13,k}$ has Weyl weight $k$.  For every $1\le k\le6$ with $k\ne4$, the previously computed exact identity gives a positive scalar-orbital decomposition of $P_{13,k}$.  Appendix~\ref{app:reconstruction-13-14} reconstructs these components exactly; the orbital coefficients of the resulting identity are recorded at \url{\CertificateRepository}.  At weight four this is impossible: Appendix~\ref{app:obstruction-orbital} gives a linear functional $L_{13}$ which is nonnegative on every scalar-orbital generator but satisfies
\[
        L_{13}(P_{13,4})<0.
\]
Thus the problem is specifically to supply an additional positive term whose weight-four component points outside the scalar-orbital cone.

The Hodge-square inequality provides such additional terms.  If $a\in H^{12}(M,\mathbb R)$, then
\[
        \int_Ma^2u^7\ge0.
\]
The coefficient search described in Appendix~\ref{app:reconstruction-13-14} shows that one such square is sufficient.  The square which occurs in the resulting identity has the form
\[
        a=\alpha u^2p_1+\beta p_3+\gamma p_1p_2+\delta p_1^3.
\]
The coefficients $\alpha,\beta,\gamma,\delta$ are not chosen in advance: they are among the unknowns in the coefficient calculation of Appendix~\ref{app:reconstruction-13-14}.

It is useful first to see exactly how such a square changes the Weyl-weight decomposition.  Put
\[
        b=\beta p_3+\gamma p_1p_2+\delta p_1^3.
\]
The term $u^2p_1$ has Weyl weight one and $b$ has Weyl weight three.  Hence
\[
\begin{aligned}
        a^2u^7
        &=(\alpha u^2p_1+b)^2u^7\\
        &=\alpha^2p_1^2u^{11}
        +2\alpha p_1b\,u^9
        +b^2u^7.
\end{aligned}
\]
Thus a positive multiple $\lambda a^2u^7$, with $\lambda>0$, changes exactly the Weyl-weight two, four and six components of $K_{13}$.  More explicitly these contributions are
\[
        \lambda\alpha^2p_1^2
\]
at weight two,
\[
        2\lambda\alpha
        \bigl(
        \delta p_1^4+\gamma p_1^2p_2+\beta p_1p_3
        \bigr)
\]
at weight four, and
\[
        \lambda
        \bigl(
        \beta p_3+\gamma p_1p_2+\delta p_1^3
        \bigr)^2
\]
at weight six.

The reason this can overcome the weight-four obstruction is already visible before doing the full calculation.  For example,
\[
        (u^2p_1+p_3)^2u^7
        =
        p_1^2u^{11}+2p_1p_3u^9+p_3^2u^7,
\]
and the weight-four contribution is $2p_1p_3$.  The separating functional of Appendix~\ref{app:obstruction-orbital} satisfies
\[
        L_{13}(2p_1p_3)=-39938<0.
\]
Thus Hodge squares contain directions on the opposite side of the separating hyperplane from the scalar-orbital cone, and can therefore repair the obstruction.

It therefore suffices to find $\lambda>0$ and $\alpha,\beta,\gamma,\delta\in\mathbb Q$ such that, after subtracting the three contributions above from $P_{13,2}$, $P_{13,4}$ and $P_{13,6}$, each remainder still admits a positive scalar-orbital decomposition; the other Weyl-weight components are unchanged.  This is a finite exact coefficient problem, solved in Appendix~\ref{app:reconstruction-13-14} with $\lambda=18$ and
\begin{equation}\label{eq:a13}
        a=
        10^{-9}\bigl(
        999999993\,u^2p_1
        +29552\,p_3
        +83790\,p_1p_2
        +74653\,p_1^3
        \bigr).
\end{equation}
The normalization of $a$ is irrelevant, since a rescaling of $a$ may be absorbed into $\lambda$; the coefficients in \eqref{eq:a13} are simply one convenient rational solution.

Since $u,p_1\in H^4(M,\mathbb R)$, $p_2\in H^8(M,\mathbb R)$ and $p_3\in H^{12}(M,\mathbb R)$, we have $a\in H^{12}(M,\mathbb R)$, and hence
\[
        \int_Ma^2u^7\ge0.
\]
Substituting the exact positive decompositions of all Weyl-weight components gives the certificate \eqref{eq:K13-restored},
\[
        K_{13}
        =
        392u^{13}
        +18\,a^2u^7
        +\text{positive scalar-orbital terms}.
\]
Every term on the right has nonnegative integral, while
\[
        \int_Mu^{13}>0.
\]
Therefore
\[
        d-14=\int_MK_{13}>0,
\]
and hence
\[
        d>14.
\]

\subsection{Quaternionic dimension fourteen}

The same idea applies in quaternionic dimension fourteen.  The decomposition uses again one Hodge square, written in the basis of the seven degree-$12$ characteristic monomials
\[
        u^3,\qquad
        u^2p_1,\qquad
        up_2,\qquad
        up_1^2,\qquad
        p_3,\qquad
        p_1p_2,\qquad
        p_1^3.
\]
Put
\[
        v=
        \begin{pmatrix}
        u^3&
        u^2p_1&
        up_2&
        up_1^2&
        p_3&
        p_1p_2&
        p_1^3
        \end{pmatrix}^{T}.
\]
Every linear combination of the entries of $v$ lies in $H^{12}(M,\mathbb R)$ and hence has nonnegative Hodge square after multiplication by $u^8$.

Consequently, if $G$ is any positive-semidefinite real symmetric $7\times7$ matrix, then
\[
        \int_M v^TGv\,u^8\ge0.
\]
Indeed, writing $G=C^TC$ gives
\[
        v^TGv=\sum_i(C_iv)^2,
\]
so this is simply a positive sum of the Hodge-square inequalities.

This quadratic form interacts with the Weyl-weight decomposition in exactly the same way as the single square above.  The seven entries of $v$ have Weyl weights
\[
        0,\qquad1,\qquad2,\qquad2,\qquad3,\qquad3,\qquad3,
\]
so a product of entries of weights $r$ and $s$, multiplied by $u^8$, has the same form as the component $P_{14,r+s}u^{14-r-s}$ of $K_{14}$.  Since $K_{14}$ has no Weyl-weight components above six, it suffices to find a positive-semidefinite rational $G$ such that, after subtracting $v^TGv\,u^8$ from $K_{14}$, every remaining Weyl-weight component has a positive scalar-orbital decomposition.  This is again a finite exact coefficient problem, solved in Appendix~\ref{app:reconstruction-13-14}.

The matrix produced by this comparison is the rank-one Gram matrix of Appendix~\ref{app:reconstruction-13-14}.  With $w$ and $G=20ww^T$ as in \eqref{eq:w14} and \eqref{eq:G14}, one has $v^TGv=20(w^Tv)^2$.  The scalar reserve that remains is the full constant term of the reduced density,
\[
        c=448>0.
\]
Together with the positive scalar-orbital terms, this is the certificate \eqref{eq:K14-restored},
\[
        K_{14}
        =
        c\,u^{14}
        +v^TGv\,u^8
        +\text{positive scalar-orbital terms}.
\]
It follows that
\[
        d-17=\int_MK_{14}>0,
\]
and therefore
\[
        d>17.
\]

Appendix~\ref{app:reconstruction-13-14} reconstructs $K_{13}$ and $K_{14}$ exactly.  The one-square expressions are \eqref{eq:K13-certificate} and \eqref{eq:K14-certificate}; their orbital coefficients are recorded at \url{\CertificateRepository}.

This establishes the present argument through quaternionic dimension fourteen, completing the proof of Theorem~\ref{thm:main}.
\appendix

\section{Exact cone obstructions at weight four}
\label{app:obstructions}

We record three short separation arguments explaining why the positivity mechanism must enlarge
at dimensions nine, eleven and thirteen.
\subsection{The projection-moment cone fails at nine and ten}
\label{app:obstruction-projection}

At Weyl weight four we use the ordered basis
\[
        b=(p_1^4,p_1^2p_2,p_2^2,p_1p_3,p_4).
\]
Recall that, in dimensions \(n=9,10\), the quartic component of the index density is a positive multiple of
\[
 F_4=
 \frac{
 175p_1^4+210p_1^2p_2+21p_2^2+80p_1p_3+18p_4
 }{87091200}.
\]

We show that \(F_4\) does not belong to the cone generated by the projection moments \(M_{4,\ell}\).

For \(n=9\), so \(r=2n+2=20\), consider the linear functional
\[
        L^{\mathrm{proj}}_9=(0,-1,84,-30,0)
\]
on the above basis.  Expanding the projection moment \(M_{4,\ell}\) using the Schur--Weyl formula gives
\[
 L^{\mathrm{proj}}_9(M_{4,\ell})
 =
 \frac{\ell(\ell-1)(20-\ell)(19-\ell)}{1017450}.
\]
Hence
\[
        L^{\mathrm{proj}}_9(M_{4,\ell})\ge0
        \qquad(1\le\ell\le20).
\]
On the other hand,
\[
        L^{\mathrm{proj}}_9(F_4)
        =-\frac{47}{4838400}<0.
\]
Thus \(F_4\) is not contained in the closed cone generated by the projection moments in dimension nine.

Similarly, for \(n=10\), so \(r=22\), define
\[
        L^{\mathrm{proj}}_{10}=(0,0,8,-3,0).
\]
Again the Schur--Weyl formula gives
\[
 L^{\mathrm{proj}}_{10}(M_{4,\ell})
 =
 \frac{\ell(\ell-1)(22-\ell)(\ell+1)}{1168860}
 \ge0
 \qquad(1\le\ell\le22),
\]
whereas
\[
        L^{\mathrm{proj}}_{10}(F_4)
        =-\frac1{1209600}<0.
\]
So \(F_4\) also lies outside the projection-moment cone in dimension ten.

Since
\[
        [K_9]_4=2^{11}F_4,
        \qquad
        [K_{10}]_4=2^{12}F_4,
\]
the new weight-four component appearing in dimensions nine and ten cannot be handled using projection moments alone.  This is why the positivity cone has to be enlarged by considering general positive-semidefinite moments.
\subsection{The PSD matrix cone fails at eleven and twelve}
\label{app:obstruction-psd}
At Weyl weight four use the ordered basis
\[
       b=(p_1^4,p_1^2p_2,p_2^2,p_1p_3,p_4).
\]
Let $M_a$ denote the invariant quartic moment obtained from a quaternionic Hermitian
PSD matrix with spectrum $a$, and put $A_j=\sum_i a_i^j$.  Define
\[
 P_{11}=(2695,-4585,25339,-1295,4361),
\]
\[
 P_{12}=(15138,-26651,150977,-6873,23971).
\]
Direct expansion of the Schur formulas gives the exact identities
\[
 P_{11}(M_a)=\frac{(9A_2-4A_1^2)^2}{55}\ge0,
 \qquad
 P_{12}(M_a)=\frac{(7A_2-3A_1^2)^2}{8}\ge0.
\]
On the quartic targets occurring in $K_{11}$ and $K_{12}$, however,
\[
       P_{11}(Q_{11,4})=-\frac{67336}{93555}<0,
       \qquad
       P_{12}(Q_{12,4})=-\frac{245414}{17325}<0.
\]
Thus neither quartic target belongs to the closed PSD matrix-moment cone.  The same argument
excludes arbitrary complex PSD test matrices: compress to the nonzero block of $X$ and average
with its quaternionic conjugate; the trace pairing with $Y$ is unchanged up to a positive scalar.

\subsection[PSD moments are mixtures of scalar orbital moments]{PSD moments are mixtures of scalar orbital moments}
\label{app:orbital-inclusion}
We show that every PSD moment, after invariant averaging, is a positive mixture of the scalar
orbital moments \eqref{eq:orbital-def}.  Consequently every linear functional which is nonnegative
on all $O_k^{(n)}(a;p)$, $a\in\RR^n_{\ge0}$, is also nonnegative on the PSD matrix-moment cone.

Let $A\succeq0$ be Hermitian on $E\oplus H$.  By Lemma~\ref{lem:weyl-connection}, $X$ and $Y$
have vanishing $H$-block, so $\tr(AY)=\tr(A_EY)$, where the compression $A_E$ of $A$ to $E$ is
again positive semidefinite.  We may therefore assume that $A$ acts on $E$.

\emph{A sum of squares independent of the curvature.}  Put $C=A^{1/2}$.  Since
$X\in\mathfrak{sp}(n)\otimes_{\RR}\Herm(E)$, we have $(CX)^*=XC$, where $^*$ acts on the matrix
coefficients and the two-form coefficients are real.  As two-forms commute,
\[
 \tr(AX^2)=\tr\bigl((CX)(XC)\bigr)
 =\sum_{i,j}(CX)_{ij}\,\overline{(CX)_{ij}}
 =\sum_{i,j}\Bigl(\bigl(\operatorname{Re}(CX)_{ij}\bigr)^2+\bigl(\operatorname{Im}(CX)_{ij}\bigr)^2\Bigr).
\]
For any matrix $N$ one has $\operatorname{Re}\tr(NX)=\tr\bigl(\tfrac12(N+N^*)X\bigr)$ and
$\operatorname{Im}\tr(NX)=\tr\bigl(\tfrac1{2i}(N-N^*)X\bigr)$.  Applying this with
$N=\epsilon_{ji}C$, where $\epsilon_{ji}$ is the matrix unit, so that $(CX)_{ij}=\tr(NX)$, gives
finitely many Hermitian matrices $B_1,\ldots,B_m$,
depending only on $A$, such that
\[
       \tr(AY)=-\tr(AX^2)=-\sum_{s=1}^m\tr(B_sX)^2.
\]
The same matrices serve for every conjugate $gXg^{-1}$, $g\in\Sp(n)$, since $gXg^{-1}$ is again
Hermitian-valued:
\[
       \tr\bigl(A\,gYg^{-1}\bigr)=-\sum_{s=1}^m\tr\bigl(B_s\,gXg^{-1}\bigr)^2.
\]

\emph{Gaussian linearization.}  Let $\xi=(\xi_1,\ldots,\xi_m)$ be independent real standard
Gaussians and put $B_\xi=\sum_s\xi_sB_s$.  For commuting variables $w_s$,
$\EE_\xi\bigl(\sum_s\xi_sw_s\bigr)^{2k}=(2k-1)!!\,\bigl(\sum_sw_s^2\bigr)^k$.  Applying this
polynomial identity to the commuting two-forms $w_s=\tr(B_s\,gXg^{-1})$ gives
\[
       \bigl(\tr(A\,gYg^{-1})\bigr)^k
       =\frac{1}{(2k-1)!!}\,\EE_\xi\Bigl[(-1)^k\tr\bigl(B_\xi\,gXg^{-1}\bigr)^{2k}\Bigr].
\]

\emph{Orbital averaging.}  Averaging over $g\in\Sp(n)$ and interchanging the two integrals, the
inner average is the scalar orbital average of Section~\ref{sec:eleven-twelve} for the fixed
matrix $B_\xi$.  As explained there, $B_\xi$ may be replaced by
$\tfrac12(B_\xi-\tau(B_\xi))\in i\mathfrak{sp}(n)$; let $a(B_\xi)\in\RR^n_{\ge0}$ be the squares
of its nonnegative eigenvalues.  With $c_{n,k}>0$ the positive normalization constant implicit in
\eqref{eq:orbital-def},
\begin{equation}\label{eq:PSD-in-orbital}
       \int_{\Sp(n)}\bigl(\tr(A\,gYg^{-1})\bigr)^k\,dg
       =\frac{c_{n,k}}{(2k-1)!!}\,\EE_\xi\,O_k^{(n)}\bigl(a(B_\xi);p\bigr).
\end{equation}
The right-hand side is an average of scalar orbital moments against a probability measure, so it
lies in the closed scalar-orbital cone.  In particular, if a linear functional $L$ on
$\mathcal V_k$ satisfies $L(O_k^{(n)}(a;p))\ge0$ for all $a\in\RR^n_{\ge0}$, then $L$ is
nonnegative on the left-hand side.

\emph{Projection moments and Gaussian terms.}  A $U(r)$-invariant polynomial in $Y$ is in
particular $\Sp(n)$-invariant.  Hence
\[
       M_{k,\ell}(Y)=\int_{U(r)}\int_{\Sp(n)}
       \bigl(\tr(gP_\ell^0g^{-1}\,hYh^{-1})\bigr)^k\,dh\,dg,
\]
and each inner integral has the form \eqref{eq:PSD-in-orbital}.  Likewise, \eqref{eq:Fk-positive-integral}
expresses $F_k$ as a limit of averages of $\bigl(\tr(A_NY)\bigr)^k$ over a $U(r)$-invariant
measure on $A_N\succeq0$, so the same argument applies, and the closed cone contains the limit.
Thus the projection moments, the Gaussian terms $F_k$, and the quaternionic PSD moments $M_a$ of
Appendix~\ref{app:obstruction-psd} all belong to the closed scalar-orbital cone.

\subsection{The scalar-orbital cone fails at thirteen and fourteen}
\label{app:obstruction-orbital}
Define
\[
 L_{13}=(32231,-53169,301931,-19969,81681),
\]
\[
 L_{14}=(33248,-94256,878742,-28008,227741).
\]
For $a\in\RR_{\ge0}^n$, with $A_j=\sum_i a_i^j$, exact expansion of
\eqref{eq:orbital-def} gives
\[
 L_{13}\bigl(O_4^{(13)}(a;p)\bigr)
 =\frac{(3A_2-A_1^2)^2}{221130}\ge0,
\]
\[
 L_{14}\bigl(O_4^{(14)}(a;p)\bigr)
 =\frac{(4A_2-A_1^2)^2}{246645}\ge0.
\]
Direct coefficient comparison gives
\[
       L_{13}(P_{13,4})<0,
       \qquad
       L_{14}(P_{14,4})<0.
\]
Hence the quartic targets lie outside the entire closed scalar-orbital cone.  Since the PSD
matrix cone is contained in the scalar-orbital cone by Appendix~\ref{app:orbital-inclusion},
adding any $M_4(A)$ cannot repair the obstruction.  This is the precise point at which a genuinely
different positivity mechanism, such as the Hodge squares of Section~\ref{sec:hodge-extension},
becomes useful.

\section{Reconstructing the index densities}
\label{app:reconstruction}

This appendix gives the exact arithmetic behind the positivity calculations in dimensions \(5\) through \(14\) provided in Sections~\ref{sec:positivity}--\ref{sec:hodge-extension}.  We first record the two universal formulas used in every dimension.  Write
$z_j=2p_j$.  The projection moments are given by the Schur--Weyl identity
\begin{equation}\label{eq:projection-schur}
 M_{k,\ell}=\sum_{\lambda\vdash k}
 \frac{f^\lambda s_\lambda(1^\ell)}{s_\lambda(1^r)}s_\lambda(Y),
\end{equation}
where $f^\lambda$ is the dimension of the irreducible $S_k$-module indexed by $\lambda$.
The Frobenius formula expresses $s_\lambda(Y)$ as a polynomial in the power sums $z_j$, so
\eqref{eq:projection-schur} is an exact rational change of coordinates.  In particular,
\[
 M_{1,\ell}=\frac{\ell}{r}z_1,
\]
\[
 M_{2,\ell}=
 \frac{\ell(\ell r-1)}{r(r-1)(r+1)}z_1^2
 +\frac{\ell(r-\ell)}{r(r-1)(r+1)}z_2.
\]

For the index density, put
\[
       \ell_j=-\frac{B_{2j}}{2j(2j)!}
\]
and define
\begin{equation}\label{eq:C-Lj}
 L_j=2\ell_j\left(n+1-2^{2j-1}
 +\sum_{v=1}^j(-1)^v\binom{2j}{2v}p_v\right).
\end{equation}
Set $A_0=1$ and
\begin{equation}\label{eq:C-Aj}
       A_j=\frac1j\sum_{v=1}^j vL_vA_{j-v}.
\end{equation}
If $r_0=\lfloor n/2\rfloor$ and $\epsilon=n\bmod2$, expand
\begin{equation}\label{eq:C-Wexp}
 (2\sinh h)^{2r_0+2}(2\cosh h)^\epsilon
 =\sum_{q\ge0}w_qh^{2q}=\sum_{q\ge0}w_qu^q.
\end{equation}
Then the reduced characteristic number is
\begin{equation}\label{eq:C-master-recurrence}
 \boxed{\qquad
        K_n^{\mathrm{red}}(p)=\sum_{j=0}^n w_{n-j}A_j(p).
 \qquad}
\end{equation}
A reduced monomial $p_\mu$ of Weyl weight $|\mu|$ represents
$p_\mu u^{n-|\mu|}$ in top degree.  Thus
\eqref{eq:C-Lj}--\eqref{eq:C-master-recurrence}, followed by the exact change of coordinates
\eqref{eq:projection-schur}, reproduce every calculation below using rational arithmetic only.

\subsection{Dimensions five through eight: projection moments}
\label{app:reconstruction-5-8}

For $n=5$ the recurrence gives
\[
 K_5^{\mathrm{red}}
 =72+\frac{536}{45}p_1+\frac49p_1^2+\frac4{45}p_2.
\]
Here $r=12$ and
\[
 M_{1,12}=2p_1,\qquad M_{2,12}=4p_1^2,\qquad
 M_{2,1}=\frac{4p_1^2+2p_2}{12\cdot13}.
\]
Substitution gives
\[
 K_5=72u^5+\frac{268}{45}M_{1,12}u^4
 +\frac1{15}M_{2,12}u^3+\frac{104}{15}M_{2,1}u^3,
\]
which after integration is exactly \eqref{eq:n5mu}.

For $n=6$ the same recurrence gives
\[
 K_6^{\mathrm{red}}
 =96+\frac{832}{45}p_1+\frac89p_1^2+\frac8{45}p_2.
\]
Now $r=14$ and
\[
 M_{1,14}=2p_1,\qquad M_{2,14}=4p_1^2,\qquad
 M_{2,1}=\frac{4p_1^2+2p_2}{14\cdot15},
\]
so
\[
 K_6=96u^6+\frac{416}{45}M_{1,14}u^5
 +\frac2{15}M_{2,14}u^4+\frac{56}{3}M_{2,1}u^4,
\]
and integration gives \eqref{eq:n6mu}.

The first cubic terms occur for $n=7,8$.  The recurrence gives
\begin{align*}
 K_7^{\mathrm{red}}={}&128+\frac{208}{7}p_1
 +\frac{8(259p_1^2+46p_2)}{945}
 +\frac{4(35p_1^3+21p_1p_2+4p_3)}{2835},\\
 K_8^{\mathrm{red}}={}&160+\frac{12896}{315}p_1
 +\frac{8(413p_1^2+71p_2)}{945}
 +\frac{8(35p_1^3+21p_1p_2+4p_3)}{2835}.
\end{align*}
For $r=16$ the three cubic moments used are
\[
 M_{3,16}=8p_1^3,\quad
 M_{3,1}=\frac{2p_1^3+3p_1p_2+p_3}{1224},\quad
 M_{3,2}=\frac{22p_1^3+15p_1p_2+2p_3}{1530},
\]
and for $r=18$ they are
\[
 M_{3,18}=8p_1^3,\quad
 M_{3,1}=\frac{2p_1^3+3p_1p_2+p_3}{1710},\quad
 M_{3,2}=\frac{2(74p_1^3+51p_1p_2+7p_3)}{14535}.
\]
Together with the quadratic formulas above, direct coefficient comparison gives
\begin{align*}
1890K_7={}&241920u^7+28080M_{1,16}u^6+668M_{2,16}u^5
 +100096M_{2,1}u^5\\
&+3M_{3,16}u^4+6528M_{3,1}u^4+4080M_{3,2}u^4,\\
4725K_8={}&756000u^8+96720M_{1,18}u^7+2710M_{2,18}u^6
 +485640M_{2,1}u^6\\
&+15M_{3,18}u^5+43320M_{3,1}u^5+29070M_{3,2}u^5.
\end{align*}
These are exactly \eqref{eq:n7mu} and \eqref{eq:n8mu} after integration.

At $n=9$ the recurrence produces a Weyl-weight-four component for the first time.  The
projection certificates just displayed therefore cease to give the required decomposition; the
next stage recognizes the new quartic term inside the larger PSD/Gaussian cone.

\subsection{Dimensions nine and ten: the Weyl quartic term}
\label{app:reconstruction-9-10}

For $n=9$ the recurrence gives
\begin{align}\label{eq:C-K9red}
 K_9^{\mathrm{red}}={}&200+\frac{90256}{1575}p_1
 +\frac{8(10634p_1^2+1641p_2)}{14175}\notag\\
 &+\frac{4(945p_1^3+509p_1p_2+78p_3)}{14175}\notag\\
 &+\frac{175p_1^4+210p_1^2p_2+21p_2^2+80p_1p_3+18p_4}{42525}.
\end{align}
For $r=20$,
\[
 M_{3,20}=8p_1^3,\quad
 M_{3,1}=\frac{2p_1^3+3p_1p_2+p_3}{2310},\quad
 M_{3,2}=\frac{2(82p_1^3+57p_1p_2+8p_3)}{21945}.
\]
Using these and the quadratic formulas gives the first seven terms in \eqref{eq:n9} exactly.
The generating series \eqref{eq:Fseries} gives
\begin{equation}\label{eq:C-F4p}
 F_4=
 \frac{175p_1^4+210p_1^2p_2+21p_2^2+80p_1p_3+18p_4}{87091200}.
\end{equation}
Since $87091200/42525=2^{11}$, the quartic line is $2^{11}F_4$.  This proves the exact
identity \eqref{eq:n9}.

For $n=10$ one obtains similarly
\begin{align*}
 K_{10}^{\mathrm{red}}={}&240+\frac{16576}{225}p_1
 +\frac{8(15073p_1^2+2217p_2)}{14175}\\
 &+\frac{16(385p_1^3+202p_1p_2+29p_3)}{14175}\\
 &+\frac{2(175p_1^4+210p_1^2p_2+21p_2^2+80p_1p_3+18p_4)}{42525}.
\end{align*}
For $r=22$,
\[
 M_{3,22}=8p_1^3,\quad
 M_{3,1}=\frac{2p_1^3+3p_1p_2+p_3}{3036},\quad
 M_{3,2}=\frac{10p_1^3+7p_1p_2+p_3}{1771}.
\]
The first three positive weights therefore give exactly the projection coefficients in
\eqref{eq:n10}, while \eqref{eq:C-F4p} identifies the last line as $2^{12}F_4$.  This is the
complete exact reconstruction of \eqref{eq:n10}.

This method cannot continue to $n=11$: Appendix~\ref{app:obstruction-psd} gives an exact
separating functional showing that the quartic targets in dimensions eleven and twelve lie
outside the entire PSD matrix-moment cone, not merely outside the particular finite family used
above.

\subsection{Dimensions eleven and twelve: scalar orbital moments}
\label{app:reconstruction-11-12}

The recurrence gives
\begin{align*}
K_{11}^{\mathrm{red}}={}&288+\frac{4965304}{51975}p_1
+\frac{152(7689p_1^2+1018p_2)}{93555}\\
&+\frac{4(92873p_1^3+44143p_1p_2+4712p_3)}{467775}
+Q_{11,4}+2^{13}F_5,\\
K_{12}^{\mathrm{red}}={}&336+\frac{6101552}{51975}p_1
+\frac{8(963501p_1^2+120259p_2)}{467775}\\
&+\frac{8(67463p_1^3+30811p_1p_2+2798p_3)}{467775}
+Q_{12,4}+2^{14}F_5,
\end{align*}
where
\[
Q_{11,4}=\frac{4(8470p_1^4+9207p_1^2p_2+825p_2^2+2882p_1p_3+450p_4)}{1403325},
\]
\[
Q_{12,4}=\frac{2(28105p_1^4+29898p_1^2p_2+2607p_2^2+8888p_1p_3+1206p_4)}{1403325},
\]
and
\[
 F_5=\frac{385p_1^5+770p_1^3p_2+440p_1^2p_3+231p_1p_2^2
 +198p_1p_4+88p_2p_3+48p_5}{11496038400}.
\]
The lower weights give the projection coefficients printed in \eqref{eq:n11} and
\eqref{eq:n12}; this is checked directly from \eqref{eq:projection-schur} exactly as in
Appendix~\ref{app:reconstruction-5-8}.

For the genuinely new quartic calculation, order the partitions as
\[
(4),(3,1),(2,2),(2,1,1),(1,1,1,1).
\]
In dimension eleven, rewriting $Q_{11,4}$ in the Schur basis gives
\begin{align*}
Q_{11,4}={}&\frac{9704}{155925}s_{(4)}+\frac{44456}{467775}s_{(3,1)}
+\frac{272}{6075}s_{(2,2)}\\
&+\frac{21104}{467775}s_{(2,1,1)}+\frac{32}{4455}s_{(1,1,1,1)}.
\end{align*}
For the five spectra in \eqref{eq:Q11orb}, the matrix of Schur values is
\[
S_{11}=\begin{pmatrix}
5&57&1510&23130&28415\\
3&47&1145&17910&26985\\
1&17&370&5616&9625\\
0&8&285&5418&9990\\
0&0&21&630&1410
\end{pmatrix}.
\]
With
\[
 t_{11}=\left(42071692,33230860,\frac{304682840}{27},
 \frac{23056120}{3},\frac{1664096}{3}\right)^T,
\]
the coefficient comparison is the single exact system $S_{11}c=t_{11}$.  Its solution is
\[
 c=\left(
 \frac{1008851484437}{2142693},
 \frac{151204920425}{714231},
 \frac{10770035384}{714231},
 \frac{51654964}{6428079},
 \frac{252882056}{1530495}
 \right)^T,
\]
which is exactly \eqref{eq:Q11orb}.

For dimension twelve the corresponding Schur expansion is
\begin{align*}
Q_{12,4}={}&\frac{15712}{155925}s_{(4)}+\frac{2944}{18711}s_{(3,1)}
+\frac{3184}{42525}s_{(2,2)}\\
&+\frac{35344}{467775}s_{(2,1,1)}+\frac{1888}{155925}s_{(1,1,1,1)}.
\end{align*}
For the five spectra in \eqref{eq:Q12orb}, the matrix of Schur values is
\[
S_{12}=\begin{pmatrix}
5&15&57&23130&25645\\
3&15&47&17910&22140\\
1&6&17&5616&7440\\
0&3&8&5418&7470\\
0&0&0&630&966
\end{pmatrix}.
\]
The five factorial factors in \eqref{eq:orbital-def} are
\[
\frac1{318073392000},\quad
\frac1{173059286400},\quad
\frac1{127876320000},\quad
\frac1{89513424000},\quad
\frac1{43609104000},
\]
so the right-hand side is
\[
 t_{12}=\left(
 \frac{1377196080}{11},\frac{319092800}{3},\frac{261804400}{7},
 26419640,2062640\right)^T.
\]
Solving $S_{12}c=t_{12}$ over $\mathbb Q$ gives
\[
 c=\left(
 \frac{2298362246}{9933},
 \frac{194400388184}{148995},
 \frac{7327263802}{9933},
 \frac{178334456}{148995},
 \frac{4485216}{3311}
 \right)^T,
\]
which is exactly \eqref{eq:Q12orb}.  Thus every coefficient in the dimension-eleven and
-twelve identities is reproduced by an explicit rational linear system.

The scalar-orbital cone is still insufficient at the next step: the separator in
Appendix~\ref{app:obstruction-orbital} is nonnegative on every scalar-orbital quartic but negative on
the quartic targets for $n=13,14$.  This is why Section~\ref{sec:hodge-extension} introduces
integrated Hodge squares.

\subsection{Dimensions thirteen and fourteen: Hodge-square certificates}
\label{app:reconstruction-13-14}

\subsubsection{Applying the recurrence}
\label{app:recurrence-13-14}

For dimensions 13 and 14, the general approach remains the same, but the coefficients become slightly unwieldy. We therefore split the exposition into two parts. On the one hand, we explicitly give the target polynomials here; for the chosen decompositions, we only give their general shape here and refer the reader to our GitHub repository (\url{\CertificateRepository}) for the exact values. The repository also provides Python programs that verify that the decompositions given produce their respective targets.

The recurrence remains the same.  It gives
\[
K_{13}^{\mathrm{red}}=392+\sum_{k=1}^6P_{13,k},
\]
with
\begin{align*}
P_{13,1}&=\frac{45727288}{315315}p_1,\\
P_{13,2}&=\frac{4(393270683p_1^2+44093313p_2)}{70945875},\\
P_{13,3}&=\frac{4(850330481p_1^3+352799421p_1p_2+18194770p_3)}{1915538625},\\
P_{13,4}&=\frac{1}{638512875}\bigl(49604555p_1^4+48174126p_1^2p_2+3778489p_2^2\\
&\hspace{22mm}+11093784p_1p_3+293502p_4\bigr),\\
P_{13,5}&=\frac{1}{1915538625}\bigl(3328325p_1^5+6076070p_1^3p_2+2902900p_1^2p_3\\
&\hspace{22mm}+1648647p_1p_2^2+945126p_1p_4+514228p_2p_3+115632p_5\bigr),\\
P_{13,6}&=\frac{1}{11493231750}\bigl(175175p_1^6+525525p_1^4p_2\\
&\hspace{22mm}+400400p_1^3p_3+315315p_1^2p_2^2\\
&\hspace{22mm}+270270p_1^2p_4+240240p_1p_2p_3+131040p_1p_5\\
&\hspace{22mm}+21021p_2^3+54054p_2p_4+22880p_3^2+33168p_6\bigr).
\end{align*}

Thus only weights $0,\ldots,6$ occur, as used in Section~\ref{subsec:hodge-thirteen}.

For the Hodge factor \eqref{eq:a13}, write
\[
a=\alpha u^2p_1+\beta p_3+\gamma p_1p_2+\delta p_1^3.
\]
Then
\[
a^2u^7=\alpha^2p_1^2u^{11}
+2\alpha p_1(\beta p_3+\gamma p_1p_2+\delta p_1^3)u^9
+(\beta p_3+\gamma p_1p_2+\delta p_1^3)^2u^7,
\]
so the exact residuals to be decomposed are $P_{13,2}$, $P_{13,4}$ and $P_{13,6}$ minus the
weight-$2$, weight-$4$ and weight-$6$ contributions stated in Section~\ref{subsec:hodge-thirteen}.

The reduced polynomial is matched by one Hodge square and twenty-nine scalar orbital
moments.  Split the factor \eqref{eq:a13} by Weyl weight as $a=u^2f_{13,1}+f_{13,3}$, where
\begin{align*}
f_{13,1}&=\frac{999999993}{10^9}p_1,\\
f_{13,3}&=\frac1{10^9}\bigl(29552\,p_3+83790\,p_1p_2+74653\,p_1^3\bigr).
\end{align*}
The square $a^2$ meets weights $2$, $4$ and $6$ only.  With coefficient $18$ on the square,
the identity in the reduced coordinates is
\begin{equation}\label{eq:K13-certificate}
K_{13}^{\mathrm{red}}
=392+\sum_{k=1}^6\sum_j c_{13,k,j}\,O_k^{(13)}(a_{13,k,j};p)
+18\bigl(f_{13,1}+f_{13,3}\bigr)^2.
\end{equation}
Every coefficient $c_{13,k,j}$ is positive.  There is one term of weight $1$, two of
weight $2$, three of weight $3$, five of weight $4$, seven of weight $5$, and eleven of
weight $6$.  Weights $1$, $3$ and $5$ receive nothing from the square.  The spectra and the
coefficients themselves are the witness recorded in
Appendix~\ref{app:recorded-witnesses}.
Restoring the factor $u^{13-k}$ on each weight-$k$ summand gives
\begin{equation}\label{eq:K13-restored}
K_{13}
=392u^{13}
+18a^2u^7
+\sum_{k=1}^6\sum_j c_{13,k,j}\,O_k^{(13)}(a_{13,k,j};p)\,u^{13-k}.
\end{equation}
The constant $392$ is the value of $K_{13}^{\mathrm{red}}$ at $p=0$.  Each orbital integrand is
nonnegative by \eqref{eq:orbital-def positivity}, and $\int_Ma^2u^7\ge0$, while $\int_Mu^{13}>0$.
Therefore
\[
d-14=\int_MK_{13}>0,
\]
so $d\ge15$.

For $n=14$ the recurrence gives
\[
K_{14}^{\mathrm{red}}=448+\sum_{k=1}^6P_{14,k},
\]
where
\begin{align*}
P_{14,1}&=\frac{816577408}{4729725}p_1,\\
P_{14,2}&=\frac{16(370709547p_1^2+38781047p_2)}{212837625},\\
P_{14,3}&=\frac{64(71758687p_1^3+28328547p_1p_2+842120p_3)}{1915538625},\\
P_{14,4}&=\frac{8}{638512875}\bigl(9204195p_1^4+8642634p_1^2p_2+648076p_2^2\\
&\hspace{22mm}+1762436p_1p_3-63807p_4\bigr),\\
P_{14,5}&=\frac{8}{1915538625}\bigl(700700p_1^5+1256255p_1^3p_2
+575575p_1^2p_3+333333p_1p_2^2\\
&\hspace{22mm}+168714p_1p_4+98527p_2p_3+12528p_5\bigr),\\
P_{14,6}&=\frac{1}{5746615875}\bigl(175175p_1^6+525525p_1^4p_2\\
&\hspace{22mm}+400400p_1^3p_3+315315p_1^2p_2^2\\
&\hspace{22mm}+270270p_1^2p_4+240240p_1p_2p_3+131040p_1p_5\\
&\hspace{22mm}+21021p_2^3+54054p_2p_4+22880p_3^2+33168p_6\bigr).
\end{align*}

We again obtain a decomposition into one square, together with twenty-nine scalar orbital moments.  Set
\begin{align*}
f_{14,1}&=p_1,\\
f_{14,3}&=\frac1{10^9}\bigl(88656\,p_3+251370\,p_1p_2+223959\,p_1^3\bigr),
\end{align*}
and write $b=u^2f_{14,1}+f_{14,3}$.  Then $b\in H^{12}(M,\mathbb R)$, so $\int_Mb^2u^8\ge0$.
In reduced coordinates,
\begin{equation}\label{eq:K14-certificate}
K_{14}^{\mathrm{red}}
=448+\sum_{k=1}^6\sum_j c_{14,k,j}\,O_k^{(14)}(a_{14,k,j};p)
+20\bigl(f_{14,1}+f_{14,3}\bigr)^2.
\end{equation}

The twenty-nine coefficients $c_{14,k,j}$ are positive, with the same weight counts
$1,2,3,5,7,11$.  The square contributes only to weights $2$, $4$ and $6$.  The spectra and
the coefficients are the second witness of Appendix~\ref{app:recorded-witnesses}.  Restoring
powers of $u$ gives

\begin{equation}\label{eq:K14-restored}
K_{14}
=448u^{14}
+20b^2u^8
+\sum_{k=1}^6\sum_j c_{14,k,j}\,O_k^{(14)}(a_{14,k,j};p)\,u^{14-k}.
\end{equation}
Since $448>0$ and $\int_Mu^{14}>0$,
\[
d-17=\int_MK_{14}>0,
\]
and therefore $d\ge18$.

In the basis $v$ of Section~\ref{sec:hodge-extension}, the factor $b$ is the linear form
$w^Tv$ with
\begin{equation}\label{eq:w14}
w=
\begin{pmatrix}
0&1&0&0&88656\cdot10^{-9}&251370\cdot10^{-9}&223959\cdot10^{-9}
\end{pmatrix}^{T}.
\end{equation}
Thus $20b^2=v^TGv$ for the rank-one Gram matrix
\begin{equation}\label{eq:G14}
G=20\,ww^T.
\end{equation}
This $G$ is positive semidefinite, and the constant term that remains is the full value
$448$ of the reduced density.  The identity above is the rational equality checked by
Algorithm~\ref{alg:verify-witness}.

These two identities complete the argument through quaternionic dimension fourteen.  No
exact positive certificate in dimension fifteen is claimed here.

\subsubsection{The recorded decompositions}
\label{app:recorded-witnesses}

Each of \eqref{eq:K13-certificate} and \eqref{eq:K14-certificate} is a sum of one positive
scalar, one positive multiple of a square, and twenty-nine positive scalar orbital moments.
The orbital terms are distributed by Weyl weight as $1,2,3,5,7,11$.  Each spectrum is represented in our supplementary data files as a
finite list of nonnegative integers, of length at most $n$ and not identically zero; a
missing coordinate means zero.  Equality of the two sides, once powers of $u$ are restored,
is a comparison of thirty monomials in $u,p_1,\ldots,p_6$.  In both dimensions, both numerators and denominators in the chosen decomposition get quite long, reaching on the order of sixty digits.

The exact lists are the files \texttt{h2\_witness\_n13.json} and
\texttt{h2\_witness\_n14.json} at
\begin{center}
\url{\CertificateRepository}.
\end{center}

Each file records the scalar, the coefficient and Weyl weight of the square, the monomials
of its factor, and, for every orbital term, the weight, the spectrum, and the rational
coefficient.  The program \texttt{verify\_h2.py} in that repository reads one such file and
runs Algorithm~\ref{alg:verify-witness}.  The reduced density and the orbital polynomials are recomputed
from \eqref{eq:C-master-recurrence} and \eqref{eq:orbital-def}, rather than read from the
file.

\begin{algorithm}[ht]
\caption{Exact comparison of a supplied dimension-$13$ or~$14$ witness with the reduced index density.}
\label{alg:verify-witness}
\begin{algorithmic}[1]
\Require $n\in\{13,14\}$.  From $W$, read a scalar $c>0$, one square given by a
coefficient $\lambda>0$, a Weyl weight $m$, and a factor $f$, together with orbital
terms $(k,a,c_{k,a})$.
\Require $2m\le n$, every $c_{k,a}>0$, every weight $1\le k\le n$, and every spectrum
$a=(a_1,\ldots,a_\ell)$ with $1\le\ell\le n$, $a_i\ge0$, and $a\ne0$.
\State Build $f$ by giving each stored monomial $p_{j_1}\cdots p_{j_r}$ the factor
$u^{m-(j_1+\cdots+j_r)}$.
\State Let $D$ be the $u$-homogenization of the reduced density produced by
\eqref{eq:C-master-recurrence}, and let $O_k^{(n)}(a;p)$ be \eqref{eq:orbital-def}.  Form
\[
T
=c\,u^n
+\lambda f^2 u^{n-2m}
+\sum c_{k,a}\,O_k^{(n)}(a;p)\,u^{n-k}.
\]
\Ensure $T=D$ as polynomials over $\mathbb Q$.  In every monomial, the power of $u$
plus the Weyl weight equals $n$.
\end{algorithmic}
\end{algorithm}

The comparison is an equality of rational polynomials, carried out in exact arithmetic. Of course, the Python scripts do nothing to verify the underlying geometric reasoning.

\subsection{Compact reproduction recipe}
\label{app:reconstruction-recipe}
For reference, the complete procedure is:
\begin{enumerate}[label=\arabic*.]
\item Compute $L_1,\ldots,L_n$ from \eqref{eq:C-Lj} and then $A_0,\ldots,A_n$ from
      \eqref{eq:C-Aj}.
\item Expand \eqref{eq:C-Wexp} through $h^{2n}$ and form $K_n^{\mathrm{red}}$ from
      \eqref{eq:C-master-recurrence}.
\item Group by Weyl weight and restore powers of $u$ only after coefficient comparison.
\item For $5\le n\le8$, use \eqref{eq:projection-schur} to solve in projection moments.
\item For $n=9,10$, use the same lower-weight calculation and identify the quartic term by
      \eqref{eq:C-F4p}.
\item For $n=11,12$, solve the displayed five-dimensional scalar-orbital systems.
\item For $n=13,14$, load the corresponding witness of
      Appendix~\ref{app:recorded-witnesses} and run Algorithm~\ref{alg:verify-witness}.
\end{enumerate}

All operations are finite exact rational calculations; no numerical approximation is ever used in
verification. Implementation is in pure Python and requires no external libraries. The code was tested under \texttt{Python 3.12.0}.

\section{Lean formalization}
\label{app:lean}

\paragraph{Goals and Overview} A mathematical proof serves two different purposes: on the one hand, it is meant to convey mathematical ideas and teach new techniques; on the other hand, it serves to certify the correctness of some mathematical claims. A study by Lamport \cite{lamport1995write} based on data by Bergman suggests that a significant fraction of mathematical proofs published in good journals contain significant gaps (which obviously need not stop them from being very successful at conveying important ideas). Therefore, we wanted to use all means available to make sure that our proofs do not succeed only in their function of transmitting knowledge, but that they also are without such gaps. 

Hence, as an additional correctness check and to support rigorous proofreading of the manuscript, we developed a partial formalization in Lean~\cite{deMouraUllrich2021} of our main result. At the time of writing, \texttt{mathlib} support for differential and algebraic geometry is still quite incomplete \cite{YuanEtAl2026}. It quickly became clear that for this reason, a complete formalization of our result (i.e. one that would only rely on results already proven in \texttt{mathlib}) would not be feasible with the resources at our disposal. As a consequence, the following decisions were made with regard to the formalization project:

\begin{enumerate}
\item The minimal goal of the formalization effort is a verification of the main result for quaternionic dimensions 1--14 against \texttt{mathlib} and a set of literature inputs comparable to those used by the present paper. 
\item The formal verification is intended to provide significant additional assurance that the arguments presented are correct beyond what the normal process of authors reviewing their own work during drafting of a submitted manuscript achieves. It is not intended to replace human review for correctness or indeed author self-review.
\item Literature inputs must enter the formalization as hypotheses to the final theorem. Compilation of the development must succeed against the three standard Lean axioms (\texttt{propext}, \texttt{Classical.choice} and \texttt{Quot.sound}).
\item The main proof target is the \emph{symmetry} statement for all complete positive quaternionic K\"ahler manifolds in quaternionic dimensions 1--14 underlying the main classification result. That means that the subsequent identification with the list of Wolf spaces \cite{Wolf1965} is outside the scope of formalization. This saves some effort, as it means that we do not have to formally construct the Wolf spaces. The price is that the formalization contains no formal example of a manifold satisfying its geometric hypotheses (see the discussion of vacuity below).
\item The project is organized in two stages, corresponding to a formal verification of the main calculations from black-boxed algebraic identities in the first stage and verification to the level described above in the second stage. The repository we published on GitHub is the endpoint of the stage 2 effort.
\item The main goal of the formalization effort is to certify the symmetry statement underlying the main classification result, \emph{not} all details of the proof published in the present paper. The rationale for this decision is that this gains us flexibility to pick a path through the proof that lends itself to formalization given the infrastructure supplied by \texttt{mathlib} and, incrementally, by the project itself; that Lean verifiers do not offer a way to \emph{verify} that the formalized \emph{proof} corresponds to the \emph{proof} presented in the present paper anyway; and the manuscript was not quite in its final form when the formalization was started.
\item Literature inputs have to be accepted facts grounded in specific papers in the literature, but the derivation from the literature does not have to be direct extraction of a quote, but may involve several ingredients via a short natural-language proof. The effect of this is that we can use results that a specialist would be able to derive from a given paper and which can be easily stated in the formalization but which may be hard to \emph{formally} prove because their formalization would require additional infrastructure not currently available in \texttt{mathlib}.
\end{enumerate}

The project produced a codebase spanning $\approx 240000$ non-empty lines of Lean code (as counted by \texttt{scc}) and verifies using only the standard axioms in Lean.

\paragraph{Assessment of reliability} The symmetry theorems currently depend on 16 literature inputs for dimensions up to 12 and on 17 literature inputs for dimensions 13 and 14. One of these, a result cited from Derdzinski and due to Hitchin (Theorem 33.4 in \cite{Derdzinski2000}), is only used in quaternionic dimension one. Manually checking the literature inputs for apparent match with the intended semantics is tedious but manageable. It does in some cases involve evaluating mathematical arguments, as not all literature sources are direct quotes from the literature. Some are relatively immediate deductions.

The same is not true, however, for object definitions that were needed on top of \texttt{mathlib} to define the object types that appear in the literature inputs or the target theorems. In total, there are around 500 of these definitions in the current codebase. We expect most of these to be routine for a reviewer who is used to reading Lean code, but not all of them are. Based on a review of a sample of 50 declarations by GPT-6-Astra, we roughly estimate that around 40--50 of the declarations would take significant work to verify (in the sample of 50, Astra classified seven as \emph{needing deep review}).

\paragraph{Vacuity} The most important failure mode that successful compilation cannot exclude is vacuity of the main theorems. This can happen in two ways. First, the definitions of the geometric objects could be unsatisfiable, so that the main theorems would apply to no manifold at all. Definitions cannot make Lean inconsistent, but they can be empty or mean something other than intended. Second, the literature inputs could be false as formalized; the main theorems would then hold vacuously, because their hypotheses can never be supplied. The first problem could be excluded by formally constructing examples of manifolds satisfying the geometric hypotheses, such as the Wolf spaces; this has not been done. The second can only be addressed by formally proving the literature inputs or by checking them against examples. The evidence against both forms of vacuity is currently indirect.

We therefore relied on the following mechanisms for checking the correctness of formalization, in particular against vacuity:

\begin{enumerate}
\item During development, the agents that wrote the Lean formalizations ran regular audits of newly written code that also checked for semantic problems.
\item At the end of the initial development phase, the resulting codebase required 48 literature sources. We ran an adversarial audit against them by letting a GPT-6-Astra agent review the codebase for several hours with the goal of proving a stronger statement than we believed reachable (LeBrun--Salamon in quaternionic dimension 15) or to find a counterexample against any of the literature sources, thereby showing inconsistency in the Lean codebase. The agent did not succeed in proving the desired higher-dimension proposition, but did find that two literature sources contained wrong claims, although it would not have been easy to formally derive a contradiction in Lean. One of the problems came from a statement in the literature that was correctly formalized but wrong, while the other came from a mismatch between a formalized statement and the literature source. Both problems were quickly patched.
\item We then eliminated literature sources one by one by formally proving the sources in Lean. This showed that at least the eliminated sources could be derived from \texttt{mathlib} and the remaining sources using reasoning that was natural enough for the GPT-6-Astra agent that ran the test to not flag a problem. We would not expect both of these conditions to be true if the \emph{eliminated} sources were wrong or wrongly formalized. Formal proof alone only shows that a source is true \emph{as formalized}, not that it matches its counterpart in the literature.
\item In a similar vein, the formalized proof of our main result is modeled on natural-language proofs of the main result. We would expect divergence of the formalized argument from the natural-language argument if the underlying concepts were wrongly formalized. This is different from blind deduction in a formal theory, where an inconsistent system of axioms proves everything, and where nobody checks whether the resulting proofs make sense semantically.
\item We checked the formal statements of the main results and spot-checked the formalizations of the literature sources and found no further problems manually.
\item Automated comparison of an almost-final version of our manuscript against the formalization also uncovered small errors in the manuscript. We would not expect this to be the case if the formalized theory was true vacuously.
\end{enumerate}

In sum, we think it is possible that our formalization contains errors, but several factors would have to conspire for the formal argument to go through in the way it does and the problems would have to be subtle enough that the language model does not find it; essentially, the problems have to be in the remaining literature inputs or definitions and they would have to still allow a massive body of normal-looking mathematical deduction to be carried out using the wrong definitions. We think the order of likelihood of this is not higher than the baseline probability that \emph{we} are wrong about the informal proof.

If there are remaining errors, we expect that they can likely be fixed, for similar general reasons as we expect that any remaining gaps in our informal exposition will be fixable.

Our overall assessment is that the formalization gives strong evidence that partly complements but does not replace our own mathematical analysis that the arguments we add on top of the literature yield the desired main results. However, even granting that the formalization itself is correct, it does not establish all intermediate results used in our proof (because it follows the same overall proof strategy but sometimes proves weaker or different results en route). Maybe most notably, the formalization does not establish the bound on the rank of the isometry group of $M$ for a minimal counterexample in all dimensions, because it does not need it.

We intend to develop the formalization further, although possibly slowly and at low priority. In that context, we consider the formal construction of the Wolf spaces up to quaternionic dimension 14 one of the most important goals, more important than further reducing the number of literature inputs, because it would exclude the first form of vacuity discussed above.

We make the Lean repository that formalizes our results available on GitHub at \url{\CertificateRepository}.

\end{document}